\documentclass{ws-m3as}

\usepackage{subfigure}
\usepackage{tikz}
\usepackage{multirow}
\usepackage{tikz-cd}
\newcommand{\dvg}{\mathrm{div}}

\begin{document}
\title{Parameter Robust Isogeometric Methods for\\ a Four-Field Formulation of Biot's Consolidation Model}

\author{Hanyu Chu}

\address{Dipartimento di Matematica, Università degli Studi di Pavia, Via Ferrata 5, 27100 \\
Pavia, Italy,\\
hanyu.chu@unipv.it}

\author{Luca Franco Pavarino}

\address{Corresponding author. Dipartimento di Matematica, Università degli Studi di Pavia, Via Ferrata 5, 27100 \\
Pavia, Italy,\\
luca.pavarino@unipv.it}

\date{}
\maketitle
\begin{abstract}
In this paper, a novel isogeometric method for Biot's consolidation model is constructed and analyzed, using a four-field formulation  where the unknown variables are the solid displacement, solid pressure, fluid flux, and fluid pressure. Mixed isogeometric spaces based on B-splines basis functions are employed in the space discretization, allowing a smooth representation of the problem geometry and solution fields. The main result of the paper is the proof of optimal error estimates that are robust with respect to material parameters for all solution fields, particularly in the case of nearly incompressible materials.
The analysis does not require a uniformly positive storage coefficient. The results of numerical experiments in two and three dimensions confirm the theoretical error estimates and high-order convergence rates attained by the proposed isogeometric Biot discretization and assess its performance with respect to the mesh size, spline polynomial degree, spline regularity, and material parameters.
\end{abstract}

\section{Introduction}
Modeling the mechanical behavior of fluid-saturated porous media has attracted considerable attention due to its extensive applications across various fields, including reservoir, environmental, and biomedical engineering. In this paper, we consider the quasi-static  Biot's consolidation model \cite{biot1941general,biot1962mechanics},
describing the interaction between
the fluid flow and deformation in an elastic porous material. The equations for
the displacement $\mathbf{u}$ and fluid pressure $p$ of an homogeneous isotropic linear elastic porous media are
\begin{align}
-\dvg(2\mu\varepsilon(\mathbf{u})+\,\lambda\,\dvg\,\mathbf{u}\mathbf{I})+\,\alpha\,\nabla p\,&=\mathbf{f_u}, \label{def:Biot problem1}\\
c_0\,p_t+\alpha\,\dvg\,\mathbf{u}_t- \dvg\,\kappa\nabla p\,&=f_p\label{def:Biot problem2}.
\end{align}
Here $\varepsilon(\mathbf{u})=\frac{1}{2}(\nabla\mathbf{u}+\nabla\mathbf{u}^T)$ is the strain from the standard linear elasticity, $\kappa$ the hydraulic conductivity, assumed to be a constant, $\alpha$ the Biot-Willis coefficient,  $c_0\geq 0$  the constrained specific storage coefficient, $\mathbf{f_u}$ the body force, and $f_p$ the source or sink term. The Lam\'e coefficients $\mu$ and $\lambda$ are assumed to be in the range $[\mu_0,\mu_1]\times[\lambda_1,\infty)$, where $0<\mu_0<\mu_1<\infty$ and $\lambda_1>0$.  To establish the well-posedness of the solution to the problem, it is necessary to impose appropriate initial and boundary conditions. Specifically, the initial conditions are given by $p(\cdot,0)=p_0$ and $\mathbf{u}(\cdot,0)=\mathbf{u}_0$. Furthermore, for clarity in presentation, the boundary conditions are assumed to be $\mathbf{u} = 0$ and $\nabla p \cdot \mathbf{n} = 0$, where $\mathbf{n}$ denotes the unit normal vector on $\partial \Omega$. When $c_0 = 0$, in order to ensure the uniqueness of the solution, we set $\int_\Omega p dx = 0$.

Several works have extensively investigated coupled finite element methods for the Biot model discretization, see e.g. \cite{Phillips2007I,Phillips2007II,Phillips2008,Huxiaozhe2017,RuiHongxing2017,YiSon-Young2013} and the references therein.
Given the multiphysics nature of the Biot model, which is typically posed in irregular domains,  the resulting discrete system presents several numerical challenges, notably elasticity locking and pressure oscillations. In order to tackle these challenges, some works developed  stabilized discretizations of the Biot model using additional fields, see e.g. \cite{Berger2015,Rodrigo2018,Preisig2011}.
The main motivations for these alternative Biot formulations and discretizations are the enhancement of the numerical stability of the resulting methods and the inclusion of specific variables, such as fluid flux,  with particular physical significance.
We refer to \cite{boffi2013mixed} for a general introduction to mixed methods.
%
 The classical two-field formulation of incompressible porous media, employing inf-sup stable finite elements for Stokes equations has been studied in \cite{Murad1992,Murad1994,Murad1996}, while a new nonsymmetric and nonconforming discretization was designed and analyzed in \cite{khan2022nonsymmetric}. A time-dependent stabilization term was added in \cite{Rodrigo2016} in order to overcome pressure oscillation. A pseudo total pressure and a new three-field formulation were proposed in \cite{Oyarzua2016,Lee2017} to obtain a stable and robust method. A novel discrete mixed form was studied in \cite{Yi2017} by applying reduced integration to the displacement bilinear form and eliminate Poisson-locking and pressure oscillations.  Decoupled algorithms for the three-field formulation of Biot problem have been studied in \cite{cai2023some,gu2023iterative}.
 Virtual element discretizations have been considered in \cite{Wang2022,Tang2021}.
Four-field Biot formulations were proposed in \cite{Korsawe2005,Tchonkova2008}  using least-squares mixed finite element methods, while \cite{YiSon-Young2014,Lee2016,qi2021four} introduced finite element methods with optimal convergence estimates for both semidiscrete and fully discrete problems.
 We remark that
Qi et al. \cite{qi2021four} have studied a four-field Biot formulation discretized with finite elements, which is
similar to the classical three-field Biot formulation (see e.g. Lee \cite{Lee2017}) where a total pressure variable is introduced. In the three-field formulation, this addition ensures the stability of the numerical scheme, but it may not be necessary in the four-field formulation. Our work instead does not rely on a total pressure variable, and by using different proofs leads to error estimates which improve the estimates in \cite{qi2021four} since they are robust with respect to the material parameters (such as the storage coefficient $c_0$ and the Lam\'e constant $\lambda$) for all solution fields.
Moreover, our work employs isogeometric instead of finite element discretizations, which we show in the numerical section to yield better accuracy per degrees of freedom, particularly in the case of isogeometric $p$- and $k$-refinements.

Isogeometric Analysis (IGA) was introduced in \cite{HCB2005} by Hughes et al. with the aim of 
extending finite element Partial Differential Equation (PDE) solvers to the splines and  non-uniform rational B-splines (NURBS) discrete spaces, widely used in Computer-Aided Design (CAD) systems, see \cite{Bazilevs2007,Bazilevs2008,Bazilevs2010,Buffa2010,Buffa2011} and the references therein.
As a generalization of classical finite element analysis, IGA
offers a more precise geometric representation of complex objects which can accurately represent many engineering shapes, including cylinders, spheres and tori. Additionally, the geometry is provided at the coarsest mesh level, directly derived from the CAD system,  eliminating geometric errors and removing the need to link the refinement procedure to a CAD representation of the geometry, as in classical finite element analysis. Another notable advantage is that, beyond the usual $h$ and $p$ refinements, in \cite{HCB2005}, the $k$-refinement is introduced, which can yield smoother functions than those used in finite element methods.
The capability of the IGA functions to achieve high-order continuity and to precisely represent complex geometries makes isogeometric analysis a suitable candidate for accurately model poroelastic media.

Isogeometric analysis has been applied to poroelasticity problems in a few previous works considering the standard two-field formulation: \cite{irzal2013isogeometric} studied numerically 1D and 2D problems, \cite{morganti2018mixed} 1D model problems, and \cite{dortdivanlioglu2018mixed} presented the proof of a stability estimate, together with a numerical validation of an inf-sup test and numerical examples in 2D.
In this paper, we construct and analyze a novel isogeometric method for a four-field formulation of Biot's poroelastic model and prove optimal error estimates that are robust with respect to material parameters for all solution fields, particularly in the case of nearly incompressible materials.
We employ the backward Euler method for the time discretization of the Biot model to obtain a fully discrete scheme.
Utilizing  stability analysis techniques from \cite{BS2013,Buffa2011}, we provide a proof of stability for our method. Additionally, a detailed error analysis is presented to illustrate that the proposed method achieves an optimal convergence order. The error estimates are uniform with respect to both the Lam\'e constant $\lambda$ and the  storage coefficient $c_0$, avoiding both elasticity locking  and pressure oscillations.

The rest of the paper is organized as follows. In Section \ref{section2}, we introduce B-splines and the geometry of the physical domain. Section \ref{section3} is dedicated to the semidiscrete problem, while the fully discrete method is presented in Section \ref{section4}. Section \ref{sec:time} presents two high-order time discretization formulations. Section \ref{section5} includes a report on some numerical experiments conducted in order to verify our results. Some conclusions are drawn in the last section.

\section{B-splines and isogeometric domain representation}
\label{section2}
In this section, we briefly recall the basic definitions of B-spline basis functions and isogeometric representations of the computational domain.
More detailed treatments of IGA can be found in \cite{Boor2001,HCB2005,Hughes2009,Schumaker2007}.

\subsection{Univariate B-splines and spaces}
Given a pair of positive integers
$p$ an $n$, a knot vector is characterized as a set of non-decreasing real numbers that represent coordinates in the parametric space of the curve,
\begin{align*}
    \Xi:=\{0=\,\xi_1,\,\xi_2,\dots,\,\xi_{n+p+1}=1\},
\end{align*}
 allowing for the repetition of knots, where $n$ is the number of the B-splines basis functions and $p$ the polynomial degree.
This implies the assumption that only $\xi_1\leq\,\xi_2\dots\,\leq\xi_{n+p+1}$. We define the vector $\mathbf{\eta}=(\eta_1,\dots,\eta_m)$ representing knots without repetitions, along with the corresponding vector
$\{r_1,\dots,\,r_m\}$
indicating knot multiplicities. Here, $r_i$ is specifically defined as the multiplicity of the knot $\eta_i$ in the set $\Xi$ and $\sum_{i=1}^m r_i=n+p+1$. We consider only open knot vectors, denoted as
 $\Xi$, where the multiplicity of a knot is at most $p$, except for the first and last knots, each with a multiplicity of $p+1$.

Associated with a knot vector $\Xi$, we define B-spline functions of degree $p$ using the well-known Cox-de Boor recursive formula. We  begin with the piecewise constants $(p=0)$
 \begin{equation}
 B_i^0(\eta)=
\left\{
\begin{aligned}
1  \quad &\text{ if } \xi_i\leq \eta\leq \xi_{i+1},\\
0\quad  &\text{ otherwise},
\end{aligned}
\right.
\end{equation}
and for $p\geq 1$, the
$p$-degree (or, equivalently, $p+1$ order)
B-spline functions $B_i^p$, for $i=1,\dots,\,n$ are constructed by the recursion
\begin{align}
B_i^p(\eta)=\frac{\eta-\xi_i}{\xi_{i+p}-\xi_i}B_i^{p-1}(\eta)+\frac{\xi_{i+P+1}-\eta}{\xi_{i+p+1}-\xi_{i+1}}B_{i+1}^{p-1}(\eta),
\end{align}
where we adopt the convention $0/0=0$. These functions yield a collection of $n$ B-splines that possess various properties, including non-negativity and the formation of a partition of unity. They also constitute a basis for the space of splines, namely, piecewise polynomials of degree $p$ with $k_j=p-r_j$ continuous derivatives at the points $\eta_j,\,j=1,\dots,m$. Hence, it is evident that $-1\leq k_j\leq p-1$, and the maximum permissible multiplicity $r_j=p+1$ results in $k_j=-1$, indicating a discontinuity at $\eta_j$. The regularity vector
$\mathbf{k}=\{k_1,\dots,\,k_m\}$ is designed to accumulate the regularity characteristics of the basis functions at the internal knots. Specifically,
$k_1=k_m=-1$ for the boundary knots, a consequence of the open knot vector structure.

The univariate spline space spanned by the B-spline basis functions $B_{i,p}$ is defined as
\begin{align}
 S_{\mathbf{k}}^p(\Xi)=\text{ span }\{B_i^p,\,i=1,\dots,n\}.
\end{align}
Given that the maximum multiplicity of internal knots is either less than or equal to the degree $p$, implying that the B-spline functions exhibit at least continuity, the derivative of each B-spline $\widehat{B}_{i,p}$ is denoted by
\begin{align}
\frac{{\rm d}B_i^p}{{\rm d}\eta} =\frac{p}{\xi_{i+p}-\xi_i}B_i^{p-1}(\eta)-\frac{p}{\xi_{i+p+1}-\xi_{i+1}}B_{i+1}^{p-1}(\eta),
\end{align}
which is made with the assumption that $B_1^{p-1}=B_{n+1}^{p-1}=0$. Notably, the resulting derivative lies within the spline space $S_{\mathbf{k}-1}^{p-1}(\Xi')$, where $\Xi'=\{\xi_2,\dots,\,\xi_{n+p}\}$ forms a $(p-1)$-open knot vector. Moreover, it is straightforward to observe that $\frac{\rm d}{{\rm d}\eta}:S_{\mathbf{k}}^p(\Xi)\to S_{\mathbf{k}-1}^{p-1}(\Xi')$ is a surjective mapping.

\subsection{Multivariate tensor product B-splines}
Given a positive integer $d$, we consider the unit cube $\widehat{\Omega}=(0,1)^d\subset \mathbb{R}^d$, denoted as the parametric domain henceforth. In generalizing the univariate splines, and given $k_l$ and $n_l$ for $l=1,\dots,\,d$, we introduce open knot vectors
\begin{align}
\Xi_l=\{\xi_{1,l},\dots,\,\xi_{n_l+p_l+{1,l}}\}
\end{align} and associated vectors $\mathbf{\eta}_l=\{\eta_{1,l}\dots,\,\eta_{m_l,l}\}$,\,$\{r_{1,l},\dots,\,r_{m_l,l}\}$ and $\mathbf{k}_l=\{k_{1,l},\dots,\,k_{m_l,l}\}$. In conjunction with these knot vectors, a parametric Cartesian mesh $\widehat{\mathcal{C}}_h$ is denoted as
\begin{align}
\widehat{\mathcal{C}}_h=\{\widehat{E}=\otimes_{l=1,\dots,\,d}(\eta_{j_l,l},\eta_{j_l+1,l}),\,1\leq j_l\leq r_l-1\}.
\end{align}
is utilized to partition the parametric domain into rectangular parallelepipeds. We denote the parametric mesh size of any element $\widehat{E}\in \widehat{\mathcal{C}}_h$ by $h_{\widehat{E}}=\text{diam}(\widehat{E})$ and the global mesh size $h=\max\{h_{\widehat{E}},\,{\widehat{E}}\in\widehat{\mathcal{C}}_h\}$.

The univariate B-spline basis functions $B_{i_l}^{p_l}$, where $i=1,\dots,\,n_l$, are linked to the provided knot vectors $\Xi_{l}$.
On the element $\widehat{E}\in\widehat{\mathcal{C}}_h$, the tensor product B-spline basis functions are defined as
\begin{align}
B_{i_1,\dots,\,i_d}^{\,p_1,\dots,\,p_d}=B_{i_1,1}^{p_1}\otimes,\dots,\otimes\,B_{i_d,d}^{p_d},\quad i_1=1,\dots,\,n_1,\,i_d=1,\dots,\,n_d.
\end{align}
Following this, the tensor-product B-spline space is defined as
\begin{align}
S_{\mathbf{k}_1,\dots,\,\mathbf{k}_d}^{p_1,\dots,\,p_d}(\Xi_1,\dots,\Xi_d)=\text{span}\,\{B_{i_1,\dots,\,i_d}^{\,p_1,\dots,\,p_d}\}_{i_1=1,\dots,\,i_d=1}^{n_1,\dots,\,n_d}.
\end{align}
The space is associated to the mesh $\widehat{\mathcal{C}}_h$, the degrees $p_l$, and the regularity vectors $\mathbf{k}_l$, $l=1,\dots,d$, as indicated in the notation. In particular, if all entries in vector $\mathbf k_l$ are equal to the integer $k_l$, $l=1,\dots,d$, then we adopt the simplified notation $S_{k_1,\dots,\,k_d}^{p_1,\dots,\,p_d}(\Xi_1,\dots,\Xi_d)$.
Consistent with univariate splines, multivariate B-spline basis functions are defined by tensor products of univariate B-splines.
They also exhibit the properties of forming a partition of unity, having local support, and maintaining non-negativity. By defining the regularity constant
$k=\min_{l=1,\dots,\,d}\min_{2\leq i_l\leq r_l-1}\{k_{i_1,l}\}$,
it is observed that the B-splines remain $C^k$ continuous across the entire domain $\widehat{\Omega}$. It is evident that the two representations, denoted as $\widehat{E}=\otimes_{l=1,\dots,\,d}(\eta_{j_l,l},\eta_{j_l+1,l})$ and $\widehat{E}=\otimes_{l=1,\dots,\,d}(\xi_{j_l,l},\xi_{j_l+1,l})$, are equivalent.
In light of this, we link each element with a support extension $\widetilde{E}$, denoted as
\begin{align}
\widetilde{E}=\otimes_{l=1,\dots,\,d}(\xi_{j_l-k_l},\,\xi_{j_l+k_l+1}).
\end{align}
Similar to the case of univariate B-splines, we also have the multivariate tensor product form for the differential operator, for example
\begin{align}
\frac{\rm d}{{\rm  d}\eta_1}:\,S_{\mathbf{k}_1,\dots,\,\mathbf{k}_d}^{p_1,\dots,\,p_d}(\Xi_1,\dots,\Xi_d)\to S_{\mathbf{k}_1-1,\dots,\,\mathbf{k}_d}^{p_1-1,\dots,\,p_d}(\Xi_1',\Xi_2,\dots,\,\Xi_d),
\end{align}
where we recall that the knot vector $\Xi_1'$\, is denoted as $\{\xi_{1,2},\dots,\,\xi_{1,n_1+p_1}\}$.

\subsection{Geometrical mapping}
Let $\Omega\subset\mathbb{R}^d$ be a single-patch physical domain characterized by a net of control points $\mathbf{c}_{i_1,\dots,\,i_d}$.
We consider
a geometrical mapping $\mathbf{F}:\widehat{\Omega}\to\Omega$ defined as a linear combination of B-spline basis functions using control points $\mathbf{c}_{i_1,\dots,\,i_d}$
\begin{align}
\mathbf{F}=\sum_{i_1=1,\dots,\,i_d=1}^{n_1,\dots,\,n_d}\mathbf{c}_{i_1,\dots,\,i_d}B_{i_1,\dots,\,i_d}^{p_1,\dots,\,p_d},
\end{align}
which acts as a parametrization for the physical domain $\Omega$ of interest.  We assume that $\mathbf{F}$
has a piecewise smooth inverse on each element $\widehat{E}\in\widehat{\mathcal{C}}_h$. We see that the parametrization $\mathbf{F}$ related to a mesh
\begin{align}
\mathcal{C}_h=\{E:\,E=\mathbf{F}(\widehat{E}),\,\widehat{E}\in\widehat{\mathcal{C}}_h\}.
\end{align}
Analogously, we define the element sizes $h_{E}=\|D\mathbf{F}\|_{L^{\infty}}h_{\widehat{E}}$, where $D\mathbf{F}$ is the Jacobian matrix of $\mathbf{F}$.

\section{Analysis of the semidiscrete IGA Biot problem}\label{section3}
In this section, we initially introduce the semidiscrete isogeometric Biot problem and subsequently provide error estimates for the semidiscrete problem. The core of the discretization is the design of a stable pair of B-splines spaces.

We first write a four-field weak formulation of the Biot problem \eqref{def:Biot problem1}-\eqref{def:Biot problem2}.
By introducing the “solid pressure” $\psi=-\lambda\,\dvg\,\mathbf{u}$, the velocity variable $\mathbf{w}=-\kappa\nabla p$, and with some algebraic manipulation, we obtain
\begin{align}
2\mu(\varepsilon(\mathbf{u}),\varepsilon(\mathbf{v}))-(\dvg\,\mathbf{v},\psi)-\alpha\,(\dvg\,\mathbf{v},p)\,&=(\mathbf{f_u},\mathbf{v})
 &&\forall\mathbf{v}\in \mathbf H_0^1(\Omega),\label{continue:Biot1}\\
 (\dvg\,\mathbf{u},\zeta)+\frac{1}{\lambda}(\psi,\zeta)\,&=0 && \forall\zeta \in L_0^2(\Omega),\label{continue:Biot2}\\
(\kappa^{-1}\mathbf{w},\mathbf{r})-(\dvg\,\mathbf{r},p)\,&=0&& \forall\mathbf{r}\in\mathbf  H_0(\dvg;\,\Omega),\label{continue:Biot3}\\
c_0(p_t,q)+\alpha(\dvg\,\mathbf{u}_t,q)+(\dvg\,\mathbf{w},q)\,&=(f_p,q)&& \forall q\in L^2(\Omega),\label{continue:Biot4}
\end{align}
and we take $q\in L_0^2(\Omega)$ if $c_0=0$.

Let $\mathbf V_h\subset \mathbf H_0^1(\Omega)$, $M_h\subset L_0^2(\Omega)$, $\mathbf W_h\subset \mathbf H_0(\dvg;\Omega)$ and $Q_h\subset L^2(\Omega)$
($Q_h\subset L_0^2(\Omega)$ if $c_0=0$) be finite dimensional isogemetric spaces defined below. Then the
semi-discrete formulation of problem \eqref{continue:Biot1}-\eqref{continue:Biot4} reads: find $\mathbf u_h\in \mathbf V_h$, $\psi_h\in M_h$, $\mathbf w_h\in \mathbf W_h$, $p_h\in Q_h$ such that
\begin{align}
2\mu(\varepsilon(\mathbf{u}_h),\varepsilon(\mathbf{v}_h))-(\dvg\,\mathbf{v}_h,\psi_h)-\alpha(\dvg\,\mathbf{v}_h,p_h)\,&=(\mathbf{f_u},\mathbf{v}_h)
 &&\forall\mathbf{v}_h\in\mathbf V_h,\label{eq:sd1}\\
 (\dvg\,{\mathbf{u}}_{h_t},\zeta_h)+\frac{1}{\lambda}({\psi}_{h_t},\zeta_h)\,&=0 && \forall\zeta_h \in M_h,\label{eq:sd2}\\
(\kappa^{-1}\mathbf{w}_h,\mathbf{r}_h)-(\dvg\,\mathbf{r}_h,p_h)\,&=0&& \forall\mathbf{r}_h\in \mathbf W_h,\label{eq:sd3}\\
c_0(p_{h_t},q_h)+\alpha(\dvg\,{\mathbf{u}}_{h_t},q_h)+(\dvg\,\mathbf{w}_h,q_h)\,&=(f_p,q_h)&& \forall q_h\in Q_h,\label{eq:sd4}
\end{align}
with given initial conditions $\mathbf u_h^0$, $p_h^0$ and $\psi_h^0$, which will be specified later.

Our main task now is to select a suitable pair of B-spline discrete spaces such that problem \eqref{eq:sd1}-\eqref{eq:sd4} is well-posed.
For concreteness, we consider here the three-dimensional case on the parametric domain $\widehat\Omega$. Let $p_p$, $k_p$, $p_{\mathbf v}$ and $k_{\mathbf v}$ be given positive integers and $\Xi_i$, $i=1,2,3$, be knot vectors.
To facilitate our subsequent stability analysis and error estimates, we assume that the parameters satisfy the following condition
\begin{equation}\label{cond:lbb}
\begin{cases}
p_{\mathbf{v}}>k_{\mathbf{v}} \geq 0,\\
p_p>k_p\geq 0,\\
p_{\mathbf{v}}-k_{\mathbf{v}}>p_p-k_p.
\end{cases}
\end{equation}
Following the notations of \cite{Buffa2010,Buffa2011}, we set
\begin{align}\label{parameteric:Xh}
\widehat{\mathbf{W}}_h=&\Big(S_{k_p+1,k_p,k_p}^{p_p+1,p_p,p_p}(\Xi_1,\,\Xi_2',\,\Xi_3')\times
S_{k_p,k_p+1,k_p}^{p_p,p_p+1,p_p}(\Xi_1',\,\Xi_2,\,\Xi_3')\notag\\
&\times S_{k_p,k_p,k_p+1}^{p_p,p_p,p_p+1}(\Xi_1',\,\Xi_2',\,\Xi_3)\Big)\cap\mathbf{H}_0(\dvg;\widehat{\Omega}),\\
\widehat{Q}_h=&S_{k_p,k_p,k_p}^{p_p,p_p,p_p}(\Xi_1',\,\Xi_2',\,\Xi_3')\,\,
(\widehat{Q}_h=S_{k_p,k_p,k_p}^{p_p,p_p,p_p}(\Xi_1',\,\Xi_2',\,\Xi_3')\cap L_0^2(\widehat\Omega)\,\text{if}\, c_0=0).
\end{align}
For $\widehat{\mathbf V}_h$, and $\widehat M_h$, we set
\begin{align}\label{parameteric:Vh}
\widehat{\mathbf{V}}_h=S_{k_{\mathbf v},k_{\mathbf v},k_{\mathbf v}}^{p_{\mathbf v},p_{\mathbf v},p_{\mathbf v}}(\Xi_1',\,\Xi_2',\,\Xi_3')\cap \mathbf{H}_0^1(\widehat\Omega)\quad\text{and}\quad
\widehat{M}_h=S_{k_p,k_p,k_p}^{p_p,p_p,p_p}(\Xi_1',\,\Xi_2',\,\Xi_3')\cap L_0^2(\widehat\Omega).
\end{align}
We also need some projection operators for spline spaces. For the one-dimensional B-spline space $S_\mathbf{k}^p$,
the projector $\widehat\Pi_{S}^p: H^1((0,1))\to S_{\mathbf k}^p(\Xi)$ is
defined by
\begin{align*}
  \widehat\Pi_{S}^pu=\sum_{i=1}^{n}(\lambda_i^pu)B_i^p,
\end{align*}
where $\lambda_i^p$ denotes dual functionals satisfying the condition $\lambda_i^p(B_j^p)=\delta_{i,j},\,i,j=1,\dots,n$, where $\delta_{i,j}$ is the standard Kronecker symbol.
The projector $\widehat\Pi_{S}^p$ has the following spline preserving property:
\begin{align*}
  \widehat\Pi_{S}^pu=u\qquad \forall u\in S_{\mathbf k}^p(\Xi).
\end{align*}

In addition to $\widehat\Pi_{S}^p$, we need to define  for any $u\in L^2((0,1))$ another projection $\widehat\Pi_{A}^{p-1}$ given by
\begin{align*}
  \widehat\Pi_{A}^{p-1}u(x)=\frac{d}{dx}\widehat \Pi_{S}^p\int_0^xu(t) dt.
\end{align*}
When the interpolation operators we consider need to preserve specific boundary values or integral averages of the function, we need to make some modifications to the above-mentioned definition. For the B-spline space $S_{\mathbf{k},0}^p:=S_\mathbf{k}^p\cap H_0^1((0,1))$, the projector
$\widehat\Pi_{0,S}^p: H_0^1((0,1))\to S_{{\mathbf k},0}^p(\Xi)$ is
given by
\begin{align*}
  \widehat\Pi_{0,S}^pu=\sum_{i=2}^{n-1}(\lambda_i^pu)B_i^p.
\end{align*}
For any $u\in L_0^2((0, 1))$, the projection $\widehat\Pi_{0,A}^{p-1}$ is defined by
\begin{align*}
  \widehat\Pi_{0,A}^{p-1}u(x)=\frac{d}{dx}\widehat \Pi_{0,S}^p\int_0^xu(t) dt.
\end{align*}
In the multivariate case, for any $\mathbf v\in \mathbf H_0^1(\widehat\Omega)$, we define
$\widehat\Pi_0^{\mathbf u}{\mathbf v}\in \widehat{\mathbf V}_h$ by
\[\widehat\Pi_0^{\mathbf u}\mathbf v=\widehat\Pi_{0,S}^{p_{\mathbf v}}\otimes \widehat\Pi_{0,S}^{p_{\mathbf v}}\otimes\widehat\Pi_{0,S}^{p_{\mathbf v}}
\mathbf v.\]
For a function $\mathbf r\in \mathbf H_0(\dvg;\widehat\Omega)$, we define $\widehat\Pi_0^{\mathbf w}\mathbf r\in \widehat{\mathbf{W}}_h$ by
\begin{equation*}
\begin{split}
\widehat\Pi_0^{\mathbf w}\mathbf r=&(\widehat\Pi_{0,S}^{p_{p+1}}\otimes \widehat\Pi_{0,A}^{p_p}\otimes\widehat\Pi_{0,A}^{p_p})\times
(\widehat\Pi_{0,A}^{p_{p}}\otimes \widehat\Pi_{0,S}^{p_p+1}\otimes\widehat\Pi_{0,A}^{p_p})\\
&\times(\widehat\Pi_{0,A}^{p_{p}}\otimes \widehat\Pi_{0,A}^{p_{p}}\otimes\widehat\Pi_{0,S}^{p_{p+1}})\mathbf r.
\end{split}
\end{equation*}
Finally, for spaces $\widehat M_h$ and $\widehat Q_h$, we define
\[\widehat\Pi_0^\psi \zeta:=\widehat\Pi_{0,A}^{p_p}\otimes \widehat\Pi_{0,A}^{p_p}\otimes\widehat\Pi_{0,A}^{p_p}\zeta\quad \forall \zeta\in L_0^2(\widehat\Omega),\]
and
\[\widehat\Pi_0^pq =\widehat\Pi_0^\psi q \quad\forall q\in L_0^2(\widehat\Omega),\,\,\text{ if } c_0=0,\]
or
\[\widehat\Pi^pq =\widehat\Pi_{A}^{p_p}\otimes \widehat\Pi_{A}^{p_p}\otimes\widehat\Pi_{A}^{p_p}q\quad
\forall q\in L^2(\widehat\Omega),\,\,\text{ if } c_0>0.\]

The interpolants defined above have the spline preserving properties (see \cite{Schumaker2007,Buffa2011}):
\begin{lemma}\label{lem:splinepreserving}
The operators $\widehat\Pi_0^{\mathbf u}$, $\widehat\Pi_0^{\psi}$, $\widehat\Pi_0^{\mathbf w}$, $\widehat\Pi_0^{p}$
$(\widehat\Pi^{p})$ are spline preserving, that is
\begin{align*}
\widehat\Pi_0^{\mathbf u}\mathbf v=\mathbf v\,~\forall \mathbf v\in  \widehat {\mathbf V}_h,\,~
\widehat\Pi_0^\psi \zeta=\zeta\,~\forall \zeta\in \widehat {M}_h, \,~
\widehat\Pi_0^{\mathbf w}\mathbf r=\mathbf r\,~\forall \mathbf r\in  \widehat {\mathbf W}_h,\,~
\widehat\Pi_0^p q=q\,~\forall q\in \widehat {Q}_h.
\end{align*}
\end{lemma}
By using the mapping $\mathbf{F}$, which is smooth together with its inverse, we can define the following pullbacks
(see \cite{Hiptmair2002})
\begin{equation}\label{transform:Sobolev}
\begin{aligned}
    &\iota^0({\mathbf v})=v\circ\mathbf{F},&& {\mathbf v}\in\mathbf H^1(\Omega),\\
      &\iota^1(\mathbf{v})=\text{det}(D\mathbf{F})(D\mathbf{F})^{-1}(\mathbf{v}\circ\mathbf{F}),&& v\in \mathbf{H}(\dvg;\,\Omega),\\
      &\iota^2(v)=\text{det}(D\mathbf{F})(v\circ\mathbf{F}),&& v\in L^2(\Omega).
\end{aligned}
\end{equation}
Furthermore, we construct the analogous discrete spaces on the physical domain $\Omega$
\[
\begin{array}{ll}
 \mathbf V_h=\{\mathbf{v}_h:\,\iota^0(\mathbf{v}_h)\in\widehat{\mathbf{V}}_h\} & \qquad M_h=\{\zeta_h:\,\iota^2(\zeta_h)\in\widehat{M}_h\} \\
 \mathbf{W}_h=\{\mathbf{w}_h:\,\iota^1(\mathbf{w}_h)\in\widehat{\mathbf{W}}_h\}& \qquad Q_h=\{q_h:\,\iota^2(q_h)\in\widehat{Q}_h\},
\end{array}
\]
and the corresponding interpolation operators
\begin{equation}\label{eq:ip}
\begin{array}{llll}
  \iota^0(\Pi_0^{\mathbf u}\mathbf v)&=&\widehat\Pi_0^{\mathbf u}(\iota^0(\mathbf v))&\qquad\forall \mathbf v\in \mathbf H_0^1(\Omega)\\
  \iota^2(\Pi_0^{\psi}\zeta)&=&\widehat\Pi_0^{\psi}(\iota^2(\zeta))&\qquad\forall \zeta\in  L_0^2(\Omega)\\
  \iota^1(\Pi_0^{\mathbf w}\mathbf r)&=&\widehat\Pi_0^{\mathbf w}(\iota^1(\mathbf r))&\qquad\forall \mathbf r\in \mathbf H_0(\dvg;\Omega)\\
  \iota^2(\Pi_0^{p}q)&=&\widehat\Pi_0^{p}(\iota^2(q))&\qquad\forall q\in  L_0^2(\Omega),
\end{array}
\end{equation}
for $c_0=0$, and the last equation is modified accordingly for $c_0>0$.
As a consequence, we can specify the initial conditions for the semi-discrete problem \eqref{eq:sd1}-\eqref{eq:sd4}, namely, $\mathbf u_h(0)=\Pi_0^{\mathbf u}\mathbf u_0$,
$p_h(0)=\Pi_0^pp_0$ and $\psi_h(0)=\Pi_0^\psi \psi_0$.

The following lemma shows the approximation properties of these discrete spaces (see \cite{Buffa2011,BBCHS2006}).
\begin{lemma}\label{lem:app}
  For $0\leq k_{\mathbf v}$, $0\leq s\leq p_{\mathrm v}$, we have
  \begin{equation}\label{ineq:appu}
    \|\mathbf v-\Pi_0^{\mathbf u}\mathbf v\|_1\leq Ch^{s-1}\|\mathbf v\|_{s+1}\qquad\forall \mathbf v\in \mathbf H^{s+1}(\Omega)\cap\mathbf H_0^1(\Omega).
  \end{equation}
  Correspondingly, for $0\leq k_p$, $0\leq s\leq p_p$, we have
  \begin{align}\label{ineq:appw}
     \|\mathbf r-\Pi_0^{\mathbf w}\mathbf r\|_{\mathbf H(\dvg;\Omega)}&\leq Ch^s\|\mathbf r\|_{\mathbf H^s(\dvg;\Omega)}&&\forall\mathbf r\in \mathbf H^s(\dvg;\Omega)\cap\mathbf H_0(\dvg;\Omega), \\
     \|\zeta-\Pi_0^\psi\zeta\|_0&\leq Ch^s\|\zeta\|_s&&\forall \zeta\in H^s(\Omega)\cap L_0^2(\Omega),\\
     \|q-\Pi^p q\|_0&\leq Ch^s\|q\|_s&&\forall q\in H^s(\Omega).
  \end{align}
\end{lemma}
Our theoretical analysis will be based on the inf-sup conditions of the following two lemmas.
\begin{lemma}
\label{lem:THLBB}
  The isogeometric Taylor-Hood pair of spline spaces $(\mathbf V_h, M_h)$, satisfies the following inf-sup condition: there exists a positive constant $C_{\mathrm{TH}}$, such that
  \begin{equation}\label{ineq:THLBB}
    \inf_{\zeta_h\in M_h}\sup_{\mathbf v_h\in\mathbf V_h}\frac{(\dvg\,\mathbf v_h,\zeta_h)}{\|\mathbf v_h\|_1\|\zeta_h\|_0}\geq C_{\mathrm{TH}},
  \end{equation}
  if the parameters condition \eqref{cond:lbb} holds.
\end{lemma}
\begin{proof}
    The proof of this lemma is modeled after the proof in \cite{BS2013} for the isogeometric discretization of the Stokes problem.
We define $M_h^\star=\{\zeta_h^\star:\iota^0(\zeta_h^\star)\in \widehat{M}_h\}$.
We associate to each macroelement $\mathcal M$ (a subset of $\mathcal T_h$) a domain $M=\cup_{E\in\mathcal M}\overline E$ and three
local spaces
\[\mathbf V_{\mathcal M}=\{\mathbf v_h\in\mathbf V_h:\text{supp}\,\mathbf v_h\subset M\},\]
\[M_{\mathcal M}^\star=\{\zeta_h^\star|_M-\frac{1}{|M|}\int_M\zeta_h^\star:\zeta_h^\star\in M_h^\star\}\]
and
\[M_{\mathcal M}=\{\zeta_h|_M-\frac{1}{|M|}\int_M\zeta_h:\zeta_h\in M_h\}.\]
The macroelement technique (see \cite{Stenberg}) applied in \cite{BS2013} states that there is a macroelement set $\mathfrak M_h$ such that for each
$\mathcal M\in \mathfrak M_h$ we have
\begin{equation}\label{ineq:macroelem}
\inf_{\zeta_h^\star\in M_{\mathcal M}^\star}\sup_{\mathbf v_h\in\mathbf V_{\mathcal M}}
\frac{(\mathbf v_h,\nabla\zeta_h^\star)}{\|\mathbf v_h\|_1|\zeta_h^\star|_h}\geq C_0,
\end{equation}
with $|\cdot|_h^2:=\sum_{E\in\mathcal M}h_E^2|\cdot|_{1,E}^2$. Moreover, \eqref{ineq:macroelem} implies that $(\mathbf V_h,M_h^\star)$ is a LBB stable
pair. To show our claim, we will prove that
\begin{equation}\label{ineq:macroelem1}
\inf_{\zeta_h\in M_{\mathcal M}}\sup_{\mathbf v_h\in\mathbf V_{\mathcal M}}
\frac{(\mathbf v_h,\nabla\zeta_h)}{\|\mathbf v_h\|_1|\zeta_h|_h}\geq C_1.
\end{equation}
 We denote $\omega=\text{det}(D\mathbf{F})$ to simplify the notation.
For any $\zeta_h\in M_{\mathcal M}$, define $\zeta_h^\star = \omega\zeta_h-\frac{1}{|M|}\int_M\omega\zeta_h$.
 Using the definition in \eqref{transform:Sobolev}, we have that $\zeta_h^\star\in M_{\mathcal M}^\star$. By the product rule and the
 Poincar\'e-Friedrichs inequality, we have the norm equivalence:
\begin{equation}\label{ineq:normeqv}
  C_\star|\zeta_h|_h\leq|\zeta_h^\star|_h \leq C^\star|\zeta_h|_h,
\end{equation}
with generic constants $C_\star, C^\star$ depending on $\omega$. Applying \eqref{ineq:macroelem} and
\eqref{ineq:normeqv}, we can select a $\mathbf v_h\in \mathbf
V_{\mathcal M}$ such that
\begin{equation}\label{ineq:TH1}
(\mathbf v_h,\nabla\zeta_h^\star)\geq C_0\|\mathbf v_h\|_1|\zeta_h^\star|_h
\geq C_0C_\star\|\mathbf v_h\|_1|\zeta_h|_h.
\end{equation}
On the other hand, we have the decomposition
\begin{equation}\label{eq:TH}
\begin{aligned}[b]
  (\mathbf v_h,\nabla\zeta_h)&=(\dvg\,\mathbf v_h,\omega^{-1}\zeta_h^\star)\\
  & = (\dvg\,\mathbf v_h,(\omega^{-1}-\overline{\omega^{-1}})\zeta_h^\star)+(\dvg\,\mathbf v_h,\overline{\omega^{-1}}\zeta_h^\star)\\
  &:=I_1+I_2,
  \end{aligned}
\end{equation}
  where $\overline{\omega^{-1}}=\int_M\omega^{-1}$. The first term $I_1$ can be bounded from below by Poincar\'e- Friedrichs inequality and \eqref{ineq:normeqv}
  \begin{equation}\label{ineq:I1}
  I_1\geq -C_2h\|\mathbf v_h\|_1|\zeta_h^\star|_h\geq -C_2C^\star h\|\mathbf v_h\|_1|\zeta_h|_h.
  \end{equation}
To bound the second term $I_2$, we use \eqref{ineq:TH1}
\begin{equation}\label{ineq:I2}
  I_2\geq \overline{\omega^{-1}}C_0C_\star\|\mathbf v_h\|_1|\zeta_h|_h.
\end{equation}
Combining \eqref{eq:TH} with \eqref{ineq:I1} and \eqref{ineq:I2}, we finally obtain \eqref{ineq:macroelem1} if $h$ is small enough. The proof is completed.
 \end{proof}

\begin{lemma}\label{lem:HdivLBB}
There exists a positive constant $C_\dvg$  such that the following inf-sup condition holds
\begin{equation}\label{ineq:dvglbb1}
\inf_{q_h\in Q_h\cap L_0^2(\Omega)}\sup_{\mathbf{w}_h\in
\mathbf W_h}\frac{(\dvg\,{\mathbf{w}_h},q_h)}{\|\mathbf{w}_h\|_{\mathbf{H}(\dvg;\Omega)}\|q_h\|_0}\geq C_\dvg.
\end{equation}
\end{lemma}
\begin{proof}
The  $\mathbf{H}(\dvg;\Omega)$ norm $\|\mathbf{w}_h\|_{\mathbf{H}(\dvg;\Omega)}$ in the denominator can be
replaced by a stronger $H^1$ norm, and the proof can be found in \cite[Proposition 5.3]{EH2013}.
\end{proof}

Now we demonstrate that our semi-discrete problem \eqref{eq:sd1}-\eqref{eq:sd4} is uniquely solvable.
\begin{lemma}\label{lem:unisob-semi}
  There exists a unique solution to problem \eqref{eq:sd1}-\eqref{eq:sd4}.
\end{lemma}
\begin{proof}
First, we introduce the operators  $\mathcal A_1:\mathbf V_h\to\mathbf V_h'$, $\mathcal A_2:M_h\to M_h'$, $\mathcal A_3:\mathbf W_h\to\mathbf W_h'$,
$\mathcal A_4:Q_h\to Q_h'$, $\mathcal B_1:\mathbf V_h\to M_h'$,  $\mathcal B_2:\mathbf V_h\to Q_h'$ and $\mathcal B_3: M_h\to Q_h'$ associated with the following bilinear forms
\[(\mathcal A_1\mathbf u_h,\mathbf v_h):=2\mu(\varepsilon(\mathbf{u}_h),\varepsilon(\mathbf{v}_h)), \quad
(\mathcal A_2\psi_h,\zeta_h):=\frac{1}{\lambda}(\psi_h,\zeta_h),\quad (\mathcal A_3\mathbf w_h,\mathbf r_h):=(\kappa^{-1}\mathbf{w}_h,\mathbf{r}_h),\]
  \[(\mathcal A_4p_h,q_h):=c_0(p_h,q_h),\quad (\mathcal B_1\mathbf v_h,\zeta_h):=(\dvg\,\mathbf{v}_h,\zeta_h),\]
  \[(\mathcal B_2\mathbf v_h,q_h):=\alpha(\dvg\,\mathbf{v}_h,q_h),\quad (\mathcal B_3\mathbf r_h,q_h):=(\dvg\,\mathbf{r}_h,q_h).\]
Therefore, problem \eqref{eq:sd1}-\eqref{eq:sd4} can be rephrased as
  \begin{align}
\mathcal A_1\mathbf u_h-\mathcal B_1^\top\psi_h-\mathcal B_2^\top p_h&=\mathcal F_1\label{eq:ode1}\\
\mathcal B_1{\mathbf u_h}_t+\mathcal A_2{\psi_h}_t&=0\label{eq:ode2}\\
\mathcal A_3\mathbf w_h-\mathcal B_3^\top p_h&=0\label{eq:ode3}\\
\mathcal A_4{p_h}_t+\mathcal B_2{\mathbf u_h}_t+\mathcal B_3\mathbf w_h&=\mathcal F_2\label{eq:ode4},
\end{align}
where $\mathcal F_1$ and $\mathcal F_2$ represent the operators corresponding to the right-hand side.
Differentiating with respect to $t$ in the first equation, and by direct calculations, we obtain
  \begin{equation}\label{eq:ode}
    (\mathcal A_4+\mathcal B_2(\mathcal A_1+\mathcal B_1^\top\mathcal A_2^{-1}\mathcal B_1)^{-1}\mathcal B_2^\top){p_h}_t+\mathcal B_3\mathcal A_3^{-1}\mathcal B_3^\top p_h=
    \mathcal F_3-\mathcal B_2(\mathcal A_1+\mathcal B_1^\top\mathcal A_2^{-1}\mathcal B_1)^{-1}\mathcal F_1.
  \end{equation}
Thanks to the Korn's inequality and Lemma \ref{lem:THLBB}, we know that $\mathcal A_4+\mathcal B_2(\mathcal A_1+\mathcal B_1^\top\mathcal A_2^{-1}\mathcal B_1)^{-1}\mathcal B_2^\top$ is
invertible. According the theory of ODE, there exists a unique $p_h$ that satisfies the above equation. In addition, $\mathbf w_h$ can be uniquely determined by equation \eqref{eq:ode3}.
 Now considering the saddle systems \eqref{eq:ode1}-\eqref{eq:ode2} for $\mathbf u_h$ and $\psi_h$, integrating \eqref{eq:ode2} in time from 0 to $t$  and using $\mathbf u_h(0)=\Pi_0^{\mathbf u}\mathbf u(0)=\mathbf0$,
$\psi_h(0)=\Pi_0^\psi\psi(0)=0$, we have $\mathcal B_1\mathbf u_{h}+\mathcal A_2\psi_{h}=0$. Since $\mathcal A_1$ is symmetric positive definite and by Lemma \ref{lem:THLBB},
we can solve uniquely for $\mathbf u_h$ and $\psi_h$.
\end{proof}

We now proceed with the error analysis for the semidiscrete problem \eqref{eq:sd1}-\eqref{eq:sd4}.
Let ($\mathbf{u},\psi, \mathbf{w}, p$) and ($\mathbf{u}_h,\psi_h, \mathbf{w}_h, p_h$) denote the solutions to problems \eqref{continue:Biot1}-\eqref{continue:Biot4} and \eqref{eq:sd1}-\eqref{eq:sd4}, respectively.
Define
\[
\begin{array}{lllll}
&\Pi_{\mathbf{u}}=\mathbf{u}-\Pi_0^\mathbf{u}\mathbf{u}, &\Pi_{\psi}=\psi-\Pi_0^{\psi}\psi,&\Pi_{\mathbf{w}}=\mathbf{w}-\Pi_0^\mathbf{w}\mathbf{w},
&\Pi_p=p-P_0^pp,\\
&\theta_{\mathbf{u}}=\Pi_0^\mathbf{u}\mathbf{u}-\mathbf{u}_h,&
\theta_\psi=\Pi_0^{\psi}\psi-\psi_h,&
\theta_{\mathbf{w}}=\Pi_0^\mathbf{w}\mathbf{w}-\mathbf{w}_h,&
\theta_p=P_0^pp-p_h,
\end{array}
\]
where $P_0^pp$ is the $L^2$ projection of $p$ onto the space $Q_h$.
Then, 
the following lemma holds.
\begin{lemma}\label{lem:SM-er-eq}
We have
\begin{align}
 2\mu(\varepsilon(\theta_{\mathbf{u}}),\varepsilon(\mathbf{v}_h))-(\dvg\,\mathbf{v}_h,\theta_\psi)-\alpha(\dvg\,\mathbf{v}_h,\theta_p)
\,&=\mathcal R_1(\mathbf v_h) &&\forall \mathbf v_h\in\mathbf V_h, \label{eq:SM-er-eq1}\\
(\dvg\,\theta_{\mathbf{u}_t},\zeta_h)+\frac{1}{\lambda}(\theta_{\psi_t},\zeta_h)\,&=\mathcal R_2(\zeta_h) && \forall\zeta_h\in M_h, \label{eq:SM-er-eq2}\\
(\kappa^{-1}\theta_{\mathbf{w}},\mathbf{r}_h)-(\dvg\,\mathbf{r}_h,\theta_p)\,&=\mathcal R_3(\mathbf r_h)&&\forall\mathbf r_h\in\mathbf W_h, \label{eq:SM-er-eq3}\\
c_0(\theta_{p_t},q_h)+\alpha(\dvg\,\theta_{\mathbf u_t},q_h)+(\dvg\,\theta_{\mathbf{w}},q_h)\,&=\mathcal R_4(q_h)&&\forall q_h\in Q_h,\label{eq:SM-er-eq4}
\end{align}
where
\begin{align*}
\mathcal R_1(\mathbf v_h):&= -2\mu(\varepsilon(\Pi_{\mathbf{u}}),\varepsilon(\mathbf{v}_h))+(\dvg\,\mathbf{v}_h,\Pi_\psi)+\alpha\,(\dvg\,\mathbf{v}_h,\Pi_p),\\
\mathcal R_2(\zeta_h):&=-(\dvg\,\Pi_{\mathbf{u}_t},\zeta_h)-\frac{1}{\lambda}(\Pi_{\psi_t},\zeta_h),\\
\mathcal R_3(\mathbf r_h):&=-(\kappa^{-1}\Pi_{\mathbf{w}},\mathbf{r}_h),\\\mathcal R_4(q_h):&=-
c_0(\Pi_{p_t},q_h)-\alpha\,(\dvg\,\Pi_{{\mathbf u}_t},q_h)-(\dvg\,\Pi_{\mathbf{w}},q_h).
\end{align*}
\end{lemma}

The following theorem presents the error estimate of the semi - discrete scheme \eqref{eq:sd1}-\eqref{eq:sd4}. Differently from \cite{qi2021four}, here we focus on the case where the storage coefficient vanishes, $c_0 = 0$, and the medium is nearly incompressible, that is, $\lambda\to\infty$.
\begin{theorem}\label{thm:semierror0}
  Let $c_0=0$, $\lambda\to\infty$, $(\mathbf{u},\psi, \mathbf{w}, p)$ and $(\mathbf{u}_h,\psi_h, \mathbf{w}_h, p_h)$ be the solutions of problems \eqref{continue:Biot1}-\eqref{continue:Biot4}
  and \eqref{eq:sd1}-\eqref{eq:sd4}, respectively. Then the following error estimate holds
  \begin{equation}\label{ineq:semierror0}
  \begin{aligned}[b]
    &\|\mathbf u-\mathbf u_h\|_{L^\infty(0,T;\mathbf H^1(\Omega))}+\|\psi-\psi_h\|_{L^2(0,T;L^2(\Omega))}\\
    +&\|\mathbf w-\mathbf w_h\|_{L^2(0,T;\mathbf L^2(\Omega))}+
    \|p-p_h\|_{L^2(0,T;L^2(\Omega))}\leq Ch^\gamma,
    \end{aligned}
  \end{equation}
  where here and in what follows  $\gamma=\min\{p_{\mathbf v},p_p+1\}$. 
\end{theorem}
\begin{proof}
  By taking $\mathbf v_h=\theta_{\mathbf u_t}$, $\zeta_h=\theta_\psi$, $\mathbf r_h=\theta_{\mathbf w}$, $q_h=\theta_p$ in equations \eqref{eq:SM-er-eq1}-\eqref{eq:SM-er-eq4}, adding
  these equations and integrating the result in time
from 0 to $t$, $t\leq T$, we obtain
  \begin{equation}\label{eq:semierr1}
   2\mu\|\varepsilon(\theta_{\mathbf{u}}(t))\|_0^2+ \frac{1}{\lambda}\|\theta_\psi(t)\|_0^2+
  \int_0^t \|\kappa^{-\frac12}\theta_{\mathbf{w}}(s)\|_0^2ds+c_0\|\theta_p(t)\|_0^2=\int_0^t\sum_{i=1}^4\mathcal R_ids.
  \end{equation}
  Let us estimate separately each term $\int_0^t\mathcal R_ids$, $i=1,\dots,4$, starting with $\int_0^t\mathcal R_1(\theta_{\mathbf u_t} )ds$. Integrating by parts, we obtain
  \begin{equation}\label{R11-1}
    \int_0^t(\varepsilon(\Pi_{\mathbf{u}}(s)),\varepsilon(\theta_{\mathbf u_t}(s)))ds=-\int_0^t(\varepsilon(\Pi_{\mathbf{u}_t}(s)),\varepsilon(\theta_{\mathbf u}(s)))ds+
    (\varepsilon(\Pi_{\mathbf{u}}(t)),\varepsilon(\theta_{\mathbf u}(t))).
  \end{equation}
Applying the Cauchy-Schwarz and  Young inequalities leads us to
\begin{equation}\label{R11-2}
\int_0^t(\varepsilon(\Pi_{\mathbf{u}}(s)),\varepsilon(\theta_{\mathbf u_t}(s)))ds \leq C\Bigg(\int_0^t\Big(\|\Pi_{\mathbf u_t}(s)\|_1^2+\|\theta_{\mathbf u}(s)\|_1^2\Big)ds
 + \|\Pi_{\mathbf u}(t)\|_1^2 \Bigg)
 +\frac14\|\theta_{\mathbf u}(t)\|_1^2.
\end{equation}
Similarly,
\begin{align}\label{R11-3}
 &\int_0^t\Big( (\dvg\,\theta_{\mathbf u_t}(s),\Pi_\psi(s))+\alpha(\dvg\,\theta_{\mathbf u_t}(s),\Pi_p(s))\Big)ds\notag\\
   \leq& C\Bigg(\int_0^t\Big(\|\Pi_{\psi_t}(s)\|_0^2+\|\Pi_{p_t}(s)\|_0^2\Big)ds+\|\Pi_{\psi}(t)\|_0^2+\|\Pi_{p}(s)\|_0^2\Bigg)+\frac14
   \|\theta_{\mathbf u}(t)\|_1^2.
\end{align}
Combining \eqref{R11-2} and \eqref{R11-3},
we find
\begin{align}\label{R11-4}
\int_0^t\mathcal R_1(\theta_{\mathbf u_t} )ds\leq&C\Bigg(\int_0^t\Big(\|\Pi_{\mathbf u_t}(s)\|_1^2+\|\theta_{\mathbf u}(s)\|_1^2+
\|\Pi_{\psi_t}(s)\|_0^2+\|\Pi_{p_t}(s)\|_0^2\Big)ds\notag\\
&+\|\Pi_{\mathbf u}(t)\|_1^2+\|\Pi_{\psi}(t)\|_0^2+\|\Pi_{p}(t)\|_0^2 \Bigg)+\frac12\|\theta_{\mathbf u}(t)\|_1^2.
\end{align}
Analogously, the other three terms can be bounded by the Cauchy-Schwarz and Young inequalities:
\begin{equation}\label{R2-1}
  \int_0^t\mathcal R_2(\theta_\psi)ds\leq C\int_0^t\Big(\|\Pi_{\mathbf u_t}(s)\|_1^2+\|\Pi_{\psi_t}(s)\|_0^2\Big)ds+\epsilon_1\int_0^t\|\theta_\psi(s)\|_0^2ds.
\end{equation}
\begin{equation}\label{R3-1}
  \int_0^t\mathcal R_3(\theta_{\mathbf w})ds\leq C\int_0^t\|\kappa^{-\frac12}\Pi_{\mathbf w}(s)\|_0^2ds+\frac12\int_0^t\|\kappa^{-\frac12}\theta_{\mathbf w}(s)\|_0^2ds.
\end{equation}
\begin{equation}\label{R_4-1}
  \int_0^t\mathcal R_4(\theta_{p})ds\leq C\int_0^t\Big(\|\Pi_{p_t}(s)\|_0^2+\|\Pi_{\mathbf u_t}(s)\|_1^2+\|\dvg\Pi_{\mathbf w}(s)\|_0^2\Big)ds
  +\epsilon_2\int_0^t\|\theta_{p}\|_0^2ds.
\end{equation}
We note that, when $\lambda\to\infty$ or $c_0=0$, the left hand side of \eqref{eq:semierr1} cannot give the effective estimates for $p$ and $\psi$. To overcome this difficulty, we need to bound
$\|\theta_\psi(t)\|_0$ and $\|\theta_p(t)\|_0$. To this end, we recall equation \eqref{eq:SM-er-eq3} and Lemma \ref{lem:HdivLBB}, finding
\begin{equation}\label{ineq:R-1}
\begin{aligned}[b]
  \|\theta_p\|_0&\leq \frac{1}{C_{\mathrm{div}}}\sup_{\mathbf r_h\in \mathbf{W}_h}\frac{(\dvg\,\mathbf r_h,\theta_p)}{\|\mathbf r_h\|_{\mathbf H(\dvg;\Omega)}}\\
  &= \frac{1}{C_{\mathrm{div}}}\sup_{\mathbf r_h\in \mathbf{W}_h}\frac{(\kappa^{-1}\theta_{\mathbf{w}},\mathbf{r}_h)-\mathcal R_3(\mathbf r_h)}{\|\mathbf r_h\|_{\mathbf H(\dvg;\Omega)}}
  \leq C_1\Big(\|\kappa^{-\frac12}\theta_{\mathbf{w}}\|_0+\|\kappa^{-\frac12}\Pi_{\mathbf w}\|_0\Big),
  \end{aligned}
\end{equation}
with a constant $C_1$ depending on $\kappa$. Likewise, we apply equation \eqref{eq:SM-er-eq1} and
Lemma \ref{lem:THLBB} to obtain
 \begin{equation}\label{ineq:R-2}
\begin{aligned}[b]
  \|\theta_\psi\|_0&\leq \frac{1}{C_{\mathrm{TH}}}\sup_{\mathbf v_h\in \mathbf{V}_h}\frac{(\dvg\,\mathbf v_h,\theta_\psi)}{\|\mathbf v_h\|_1}\\
&  = \frac{1}{C_{\mathrm{TH}}}\sup_{\mathbf v_h\in \mathbf{V}_h}\frac{2\mu(\varepsilon(\theta_{\mathbf{u}}),\varepsilon(\mathbf{v}_h))-\alpha(\dvg\,\mathbf{v}_h,\theta_p)-
\mathcal R_1(\mathbf v_h)}{\|\mathbf v_h\|_1}\\
  &\leq C_2\Big(\|\theta_{\mathbf u}\|_1+\|\Pi_{\mathbf u}\|_1+\|\Pi_\psi\|_0+\|\Pi_p\|_0
  +\|+\|\kappa^{-\frac12}\theta_{\mathbf{w}}\|_0+\|\kappa^{-\frac12}\Pi_{\mathbf w}\|_0\Big),
  \end{aligned}
\end{equation}
where we have used \eqref{ineq:R-1} for the last inequality.
Taking $\epsilon_1=\frac1{4C_1}$, $\epsilon_2=\frac{1}{4C_2}$, plugging \eqref{R11-4}-\eqref{R_4-1} into \eqref{eq:semierr1} and using \eqref{ineq:R-1}, \eqref{ineq:R-2}, we obtain,
\begin{equation}\label{ineq:semiR}
  \begin{aligned}[b]
  &\|\varepsilon(\theta_{\mathbf{u}}(t))\|_0^2+\int_0^t \|\kappa^{-\frac12}\theta_{\mathbf{w}}(s)\|_0^2ds\\
  \leq &C\int_0^t\|\varepsilon(\theta_{\mathbf{u}}(s))\|_0^2ds+
C\int_0^t\Bigg(\|\Pi_{\mathbf u_t}(s)\|_1^2+\|\Pi_{\psi_t}(s)\|_0^2+\|\Pi_{p_t}(s)\|_0^2\\&+\|\Pi_{\mathbf u}(s)\|_1^2
+\|\Pi_\psi(s)\|_0^2
+\|\Pi_p(s)\|_0^2+\|\kappa^{-\frac12}\Pi_{\mathbf w}(s)\|_0^2+\|\dvg\Pi_{\mathbf w}(s)\|_0^2\Bigg)ds\\
&+C\Big(\|\Pi_{\mathbf u}(t)\|_1^2+\|\Pi_{\psi}(t)\|_0^2+\|\Pi_{p}(t)\|_0^2\Big).
  \end{aligned}
\end{equation}
Using \eqref{ineq:semiR}, Korn  inequality,  Gronwall inequality and Lemma \ref{lem:app} we derive the bound
\begin{equation}\label{ineq:semiR-1}
\begin{aligned}
 &\sup_{0\leq t\leq T} \|\theta_{\mathbf{u}}(t)\|_1^2+\int_0^t \|\theta_{\mathbf{w}}(s)\|_0^2ds\\
 \leq&C h^{2\gamma}\int_0^T\Bigg(\|\mathbf u_t(s)\|_{\gamma+1}^2+\|p_t(s)\|_{\gamma}^2+\|\psi_t(s)\|_{\gamma}^2+
 \|\mathbf u(s)\|_{\gamma+1}^2+\|p(s)\|_{\gamma}^2\\&
 +\|\mathbf w(s)\|_{\mathbf H^s(\dvg;\Omega)}+C\|\psi(s)\|_{\gamma}^2\Bigg)ds+Ch^{2\gamma}\sup_{0\leq t\leq T}\Big(\|\mathbf u(t)\|_{\gamma+1}^2
+\|p(t)\|_{\gamma}^2+\|\psi(t)\|_{\gamma}^2\Big).
 \end{aligned}
\end{equation}
Finally, we use \eqref{ineq:R-1}, \eqref{ineq:R-2} and \eqref{ineq:semiR-1} to complete the proof.
\end{proof}
When $c_0\geq\tau_0>0$, $\lambda<\tau_1<\infty$, with the same techniques used above we can obtain the following result.
\begin{theorem}\label{thm:semierror}
  Let $c_0\geq\tau_0>0$, $\lambda<\tau_1<\infty$, $(\mathbf{u},\psi, \mathbf{w}, p)$ and $(\mathbf{u}_h,\psi_h, \mathbf{w}_h, p_h)$ be the solutions to problems \eqref{continue:Biot1}-\eqref{continue:Biot4}
  and \eqref{eq:sd1}-\eqref{eq:sd4}, respectively. Then, the following error estimate holds
  \begin{equation}\label{ineq:semierror}
  \begin{aligned}[b]
    &\|\mathbf u-\mathbf u_h\|_{L^\infty(0,T;\mathbf H^1(\Omega))}+\|\psi-\psi_h\|_{L^\infty(0,T;L^2(\Omega))}\\
    +&\|\mathbf w-\mathbf w_h\|_{L^2(0,T;\mathbf L^2(\Omega))}+
    \|p-p_h\|_{L^\infty(0,T;L^2(\Omega))}\leq Ch^\gamma.
    \end{aligned}
  \end{equation}
\end{theorem}

\section{Analysis of the fully discrete IGA Biot problem}\label{section4}
Let $\Delta t=T/N$ be a uniform time step size, and $t^n=n\Delta t,\,0\leq n\leq N$, be the discrete time instants. Using the backward Euler method in time and IGA in space, we obtain a fully discrete version of Biot problem \eqref{continue:Biot1}-\eqref{continue:Biot4}: given $\mathbf u_h^0=\Pi_0^{\mathbf u}\mathbf u_0$,
$p_h^0=\Pi_0^pp_0$ and $\psi_h^0=\Pi_0^\psi \psi_0$, find $(\mathbf{u}_h^n,\psi_h^n,\mathbf{w}_h^n,p_h^n)\in\mathbf{V}_h\times M_h\times\mathbf{W}_h\times Q_h$ such that
\begin{align}
2\mu(\varepsilon(\mathbf{u}_h^n),\varepsilon(\mathbf{v}_h))-(\dvg\,\mathbf{v}_h,\psi_h^n)-\alpha(\dvg\,\mathbf{v}_h,p_h^n)
\,&=(\mathbf{f}_{\mathbf{u}}^n,\mathbf{v}_h),\label{fulldiscrete:Biot1}\\
 \left(\frac{\dvg\,\mathbf{u}_h^n-\dvg\,\mathbf{u}_h^{n-1}}{\Delta t},\zeta_h\right)+\frac{1}{\lambda}\left(\frac{\psi_h^n-\psi_h^{n-1}}{\Delta t},\zeta_h\right)\,&=0,\label{fulldiscrete:Biot2}\\
(\kappa^{-1}\mathbf{w}_h^n,\mathbf{r}_h)-(\dvg\,\mathbf{r}_h,p_h^n)\,&=0,\label{fulldiscrete:Biot3}\\
c_0\left(\frac{p_h^n-p_h^{n-1}}{\Delta t},q_h\right)+\alpha\left(\frac{\dvg\,\mathbf{u}_h^n-\dvg\,\mathbf{u}_h^{n-1}}{\Delta t},q_h\right)+(\dvg\,\mathbf{w}_h^n,q_h)\,&=(f_p^n,q_h)\label{fulldiscrete:Biot4}
\end{align}
for every $(\mathbf{v}_h,\zeta_h,\mathbf{r}_h,q_h)\in\mathbf{V}_h\times M_h\times\mathbf{W}_h\times Q_h$, where $\mathbf{u}_h^n=\mathbf{u}_h(:,t^n)$, $p_h^n=p_h(:,t^n)$.
\subsection{Existence and uniqueness}
We consider the existence and uniqueness of the following linear system derived from \eqref{fulldiscrete:Biot1}-\eqref{fulldiscrete:Biot4}:
find $(\mathbf{u}_h^n,\psi_h^n,\mathbf{w}_h^n,p_h^n)\in\mathbf{V}_h\times M_h\times\mathbf{W}_h\times Q_h$ such that
\begin{align}
2\mu(\varepsilon(\mathbf{u}_h^n),\varepsilon(\mathbf{v}_h))-(\dvg\,\mathbf{v}_h,\psi_h^n)-\alpha(\dvg\,\mathbf{v}_h,p_h^n)
\,&=(\mathbf{f}_{\mathbf{u}}^n,\mathbf{v}_h),\label{derive:Biot1}\\
 (\dvg\,\mathbf{u}_h^n,\zeta_h)+\frac{1}{\lambda}(\psi_h^n,\zeta_h)\,&=(\widetilde{f}_{\psi},\,\zeta_h),\label{derive:Biot2}\\
\Delta t(\kappa^{-1}\mathbf{w}_h^n,\mathbf{r}_h)-\Delta t(\dvg\,\mathbf{r}_h,p_h^n)\,&=0,\label{derive:Biot3}\\
c_0(p_h^n,q_h)+\alpha(\dvg\,\mathbf{u}_h^n,q_h)+\Delta t(\dvg\,\mathbf{w}_h^n,q_h)\,&=(\widetilde{f}_p,q_h),\label{derive:Biot4}
\end{align}
with $(\mathbf{v}_h,\zeta_h,\mathbf{r}_h,q_h)\in\mathbf{V}_h\times M_h\times\mathbf{W}_h\times Q_h$,
where we use the notations $\widetilde{f}_p=\Delta tf_p(t^n)+c_0\,p_h^{n-1}+\alpha\,\dvg\,\mathbf{u}_h^{n-1}$ and $\widetilde{f}_{\psi}=\dvg\,\mathbf{u}_h^{n-1}+\frac{1}{\lambda}\psi_h^{n-1}$.
We define the bilinear form
\begin{align*}
B_h&(\mathbf{u}_h,\psi_h,\mathbf{w}_h,p_h;\mathbf{v}_h,\zeta_h,\mathbf{r}_h,q_h)\\
&=2\mu(\varepsilon(\mathbf{u}_h),\varepsilon(\mathbf{v}_h))
-(\dvg\,\mathbf{v}_h,\psi_h)-\alpha(\dvg\,\mathbf{v}_h,p_h)
+(\dvg\,\mathbf{u}_h,\zeta_h)
+\frac{1}{\lambda}(\psi_h,\zeta_h)\\
&+\Delta t(\kappa^{-1}\mathbf{w}_h,\mathbf{r}_h)
-\Delta t(\dvg\,\mathbf{r}_h,p_h)+c_0(p_h,q_h)
+\alpha(\dvg\,\mathbf{u}_h,q_h)+\Delta t(\dvg\,\mathbf{w}_h,q_h).
\end{align*}

We are in a position to demonstrate the existence and uniqueness of the solution to the fully discrete problem.
In previous papers based on a three-field Biot formulation, the total pressure variable was introduced, because a stable LBB condition is necessary for the stability of the system. However, a $(\dvg,\cdot)$ term is not enough to bound both the pressure and the solid pressure simultaneously,  see \cite{Lee2017} for specific details. Therefore, we will use different arguments as follows.
\begin{lemma}\label{lem:unisob:full}
  Let $(\mathbf u_h^{n-1}, \psi_h^{n-1}, \mathbf w_h^{n-1}, p_h^{n-1})$, $1\leq n\leq N$, be given.  Then there exists a unique solution $(\mathbf u_h^n, \psi_h^n, \mathbf r_h^n, p_h^n)$ of the
  fully discrete IGA Biot formulation \eqref{fulldiscrete:Biot1}-\eqref{fulldiscrete:Biot4}.
\end{lemma}
\begin{proof}
We only consider $c_0=0$, while the case $c_0>0$ is similar.
Since our discrete spaces are finite dimensional, it is sufficient to prove that the equation 
\[B_h(\mathbf{u}_h,\psi_h,\mathbf{w}_h,p_h;\mathbf{v}_h,\zeta_h,\mathbf{r}_h,q_h)=0,\]
for any  $(\mathbf{v}_h,\zeta_h,\mathbf{r}_h,q_h)\in\mathbf{V}_h\times M_h\times\mathbf{W}_h\times Q_h$,
has only the zero solution.
With this in view, we first apply Lemma \ref{lem:THLBB} and obtain that for any $\psi_h\in M_h$ and $p_h\in Q_h$, there exists a $\mathbf{z}_h\in\mathbf{V}_h$ such that
    \begin{align}\label{eq:eu1}
        (\dvg\,\mathbf{z}_h,\psi_h+\alpha p_h)\geq\beta_v\|\psi_h+\alpha p_h\|_{0,\Omega}^2,
        \quad\quad\|\mathbf{z}_h\|_{1,\Omega}=\|\psi_h+\alpha p_h\|_0.
    \end{align}
In addition, Lemma \ref{lem:HdivLBB} implies that, for any $p_h\in Q_h$, there exists a $\mathbf{J}_h\in \mathbf W_h$ satisfying
    \begin{align}\label{eq:eu2}
        (\dvg\,\mathbf{J}_h,p_h)\geq\beta_p\|p_h\|_{0,\Omega}^2,
        \quad \quad\|\mathbf{J}_h\|_{\mathbf H(\dvg;\Omega)}=\|p_h\|_{0,\Omega}.
    \end{align}
By taking $\mathbf{v}_h=\mathbf{u}_h-\theta_1\mathbf{z}_h,\,\zeta_h=\psi_h,\,q_h=p_h+\theta_2\,\Delta t\,\dvg\,\mathbf{w}_h, \,\mathbf{r}_h=\mathbf{w}_h-\theta_3\mathbf{J}_h$ in the  homogeneous linear
system
\begin{equation}\label{eq:eu3}
  B_h(\mathbf{u}_h,\psi_h,\mathbf{w}_h,p_h;\mathbf{v}_h,\zeta_h,\mathbf{r}_h,q_h)=0\quad\forall (\mathbf{v}_h,\zeta_h,\mathbf{r}_h,q_h)\in\mathbf V_h\times M_h\times \mathbf W_h\times Q_h,
\end{equation}
we obtain
\begin{equation}\label{eq:eu4}
      \begin{aligned}[b]
0=&B_h(\mathbf{u}_h,\psi_h,\mathbf{w}_h,p_h;\mathbf{v}_h,\zeta_h,\mathbf{r}_h,q_h)\\
=&2\,\mu\,(\varepsilon(\mathbf{u}_h),\varepsilon(\mathbf{u}_h))-2\,\theta_1\,\mu\,(\varepsilon(\mathbf{u}_h),\varepsilon(\mathbf{z}_h))+\theta_1\,(\dvg\,\mathbf{z}_h,\psi_h+\alpha \,p_h)+\frac{1}{\lambda}\,(\psi_h,\psi_h)\\
&+\Delta t\,(\kappa^{-1}\,\mathbf{w}_h,\mathbf{w}_h)-\theta_3\,\Delta t\,(\kappa^{-1}\mathbf{w}_h,\mathbf{J}_h)+\theta_3\,\Delta t\,(\dvg\,\mathbf{J}_h,p_h)+c_0\,(p_h,p_h)\\
&+c_0\,\theta_2\,\Delta t\,(p_h,\dvg\,\mathbf{w}_h)
+\alpha\,\theta_2\,\Delta t\,(\dvg\,\mathbf{u}_h,\dvg\,\mathbf{w}_h)+\theta_2\,(\Delta t)^2\,(\dvg\,\mathbf{w}_h,\dvg\,\mathbf{w}_h)\\
:&=\Lambda.
\end{aligned}
\end{equation}
By \eqref{eq:eu1}, \eqref{eq:eu2} and applying the Cauchy-Schwarz and Young's inequalities, we bound $\Lambda$ from below as follows
\begin{equation}\label{ineq:eu5}
\begin{aligned}[b]
\Lambda
&\geq 2\,\mu\|\varepsilon(\mathbf{u}_h)\|_0^2-\frac{\theta_1\epsilon_1}{2}\|\varepsilon(\mathbf{u}_h)\|_0^2-\frac{\theta_1}{2\,\epsilon_1}\,\|\varepsilon(\mathbf{z}_h)\|_0^2+\theta_1\,\beta_v\,\|\psi_h+\alpha\, p_h\|_0^2+\beta_0\,\|\psi_h\|_0^2\\
&+\beta_1\,\Delta t\,\|\mathbf{w}_h\|_0^2-\theta_3\,\epsilon_2\,\Delta t\,\|\mathbf{w}_h\|_0^2-\frac{\theta_3}{\epsilon_2}\|\mathbf{J}_h\|_0^2+\theta_3\,\Delta t\,\beta_p\,\|p_h\|_0^2+c_0\,\|p_h\|_0^2\\
&-\theta_2\,c_0\,\epsilon_3\,\|p_h\|_0^2
-\frac{\theta_2\,c_0}{\epsilon_3}\,(\Delta t)^2\|\dvg\,\mathbf{w}_h\|_0^2-\frac{\alpha\,\theta_2\,\epsilon_4}{2}\,\|\dvg\,\mathbf{u}_h\|_0^2\\&-\frac{\alpha\,\theta_2}{2\,\epsilon_4}\,(\Delta t)^2\|\dvg\,\mathbf{w}_h\|_0^2+\theta_2\,(\Delta t)^2\|\dvg\,\mathbf{w}_h\|_0^2\\
\geq& \left[\beta_\mu(2-\frac{\theta_1\epsilon_1}{2})-\alpha\,\theta_2\,\epsilon_4\right]\|\mathbf{u}_h\|_1^2+(\beta_2+\beta_0)\|\psi_h\|_0^2+(\beta_1-\theta_3\,\epsilon_2)\Delta t\,\|\mathbf{w}_h\|_0^2\\
&+\left(\theta_2-\frac{\theta_2\,c_0}{\epsilon_3}-\frac{\alpha\,\theta_2}{2\,\epsilon_4}\right)(\Delta t)^2\|\dvg\,\mathbf{w}_h\|_0^2
+\left(\beta_2-\frac{\theta_3}{\epsilon_2}+c_0-\theta_2\,c_0\,\epsilon_3\right)\|p_h\|_0^2\\
:=&\,C_1\|\mathbf{u}_h\|_1^2+C_2\|\psi_h\|_0^2+C_3\|\mathbf{w}_h\|_0^2+C_4\|\dvg\,\mathbf{w}_h\|_0^2+C_5\|p_h\|_0^2.
    \end{aligned}
\end{equation}
In the final step, we have denoted by $C_i$, $i=1,\dots,5$, the positive constants obtained by appropriately selecting the constants $\epsilon_i$, $i=1,\dots, 4$ and $\theta_i$, $i=1,2,3$.
Finally, by \eqref{ineq:eu5} and \eqref{eq:eu4} we conclude that $(\mathbf u_h, \psi_h, \mathbf w_h, p_h)$
are all equal to zero or zero vectors. The proof is completed.
\end{proof}

\subsection{A priori error estimates}
The poroelastic locking commonly arises when the specific storage term under constraint is zero $(c_0=0)$, the permeability of the porous medium is extremely low, and a small time step is employed. Following the approach outlined in \cite{Yi2017}, we now derive a priori error estimates for $c_0=0$,
focusing  our error analysis on the fully discrete scheme \eqref{fulldiscrete:Biot1}-\eqref{fulldiscrete:Biot4}.
By a Taylor expansion, we have
\begin{align}\label{eq:Taylor}
\frac{\mathbf{u}^n-\mathbf{u}^{n-1}}{\Delta t}=\mathbf{u}_t^n+\frac{1}{\Delta t}\int_{t^{n-1}}^{t^n}\left(s-t^{n-1}\right)\mathbf{u}_{tt}(s) \,{\rm d}s.
\end{align}
Considering the exact solution $(\mathbf{u},\psi,\mathbf{w},p)$ at time $t=t^n$ and employing test functions $(\mathbf{v}_h,\zeta_h,\mathbf{r}_h,q_h)\in \mathbf{V}_h\times M_h\times\mathbf{W}_h\times Q_h$ in equation
\eqref{continue:Biot1}-\eqref{continue:Biot4}, we can use  \eqref{eq:Taylor} to derive the following formulation:
\begin{align}
2\mu(\varepsilon(\mathbf{u}^n),\varepsilon(\mathbf{v}_h))-(\dvg\,\mathbf{v}_h,\psi^n)-\alpha(\dvg\,\mathbf{v}_h,p^n)
\,&=(\mathbf{f}_{\mathbf{u}}^n,\mathbf{v}_h),\label{Taylorexpansion:Biot1}\\
\left(\frac{\dvg\,\left(\mathbf{u}^n-\mathbf{u}^{n-1}\right)}{\Delta t},\zeta_h\right)+\frac{1}{\lambda}\left(\frac{\psi^n-\psi^{n-1}}{\Delta t},\zeta_h\right)\, &=\Upsilon_1+\Upsilon_2,
\label{Taylorexpansion:Biot2}\\
(\kappa^{-1}\mathbf{w}^n,\mathbf{r}_h)-(\dvg\,\mathbf{r}_h,p_h^n)\,&=0,\label{Taylorexpansion:Biot3}\\
\alpha\left(\frac{\dvg(\mathbf{u}^n-\mathbf{u}^{n-1})}{\Delta t},q_h\right)+(\dvg\,\mathbf{w}^n,q_h)\,&=(f_p^n,q_h)+\Upsilon_3,\label{Taylorexpansion:Biot4}
\end{align}
where
\begin{align}
\Upsilon_1&=\frac{1}{\Delta t}\left(\int_{t^{n-1}}^{t^n}\left(s-t^{n-1}\right)\dvg\,\mathbf{u}_{tt}(s) \,{\rm d}s, \zeta_h\right)\label{Upsilon1},\\
\Upsilon_2&=\frac{1}{\lambda\,\Delta t}\left(\int_{t^{n-1}}^{t^n}\left(s-t^{n-1}\right)\psi_{tt}(s)\,{\rm d}s, \zeta_h\right),\label{Upsilon2}\\
\Upsilon_3&=\frac{\alpha}{\Delta t}\left(\int_{t^{n-1}}^{t^n}\left(s-t^{n-1}\right)\dvg\,\mathbf{u}_{tt}(s) \,{\rm d}s, q_h\right).\label{Upsilon3}
\end{align}

To simplify the notation, we introduce the following quantities at each discrete time $t^n$:
\[
\begin{array}{lllll}
&\Pi_{\mathbf{u}}^n=\mathbf{u}^n-\Pi_0^\mathbf{u}\mathbf{u}^n, & \quad\Pi_{\psi}^n=\psi^n-\Pi_0^{\psi}\psi^n,&\quad\Pi_{\mathbf{w}}^n=\mathbf{w}^n-\Pi_0^\mathbf{w}\mathbf{w}^n,
&\quad\Pi_p^n=p^n-P_0^pp^n,\\
&\theta_{\mathbf{u}}^n=\Pi_0^\mathbf{u}\mathbf{u}^n-\mathbf{u}_h^n,&
\quad\theta_\psi^n=\Pi_0^{\psi}\psi^n-\psi_h^n,&
\quad\theta_{\mathbf{w}}^n=\Pi_0^\mathbf{w}\mathbf{w}^n-\mathbf{w}_h^n,&
\quad\theta_p^n=P_0^pp^n-p_h^n.
\end{array}
\]
We have the following auxiliary error equations analogous to the ones in Lemma \ref{lem:SM-er-eq}:
  \begin{align}
2\mu(\varepsilon(\theta_{\mathbf{u}}^n),\varepsilon(\mathbf{v}_h))-(\dvg\,\mathbf{v}_h,\theta_{\psi}^n)-\alpha(\dvg\,\mathbf{v}_h,\theta_{p}^n)
&\,=\mathcal R_1^n(\mathbf v_h) &&\forall \mathbf v_h\in\mathbf V_h, \label{eq:Fu-er-eq1}\\
\left(\frac{\dvg\,(\theta_{\mathbf{u}}^n-\theta_{\mathbf{u}}^{n-1})}{\Delta t},\zeta_h\right)+\frac{1}{\lambda}\left(\frac{\theta_{\psi}^n-\theta_{\psi}^{n-1}}{\Delta t},\zeta_h\right)
 &\,=\mathcal R_2^n(\zeta_h) && \forall\zeta_h\in M_h, \label{eq:Fu-er-eq2}\\
(\kappa^{-1}\theta_{\mathbf{w}}^n,\mathbf{r}_h)-(\dvg\,\mathbf{r}_h,\theta_{p}^n)&\,=\mathcal R_3^n(\mathbf r_h)&&\forall\mathbf r_h\in\mathbf W_h, \label{eq:Fu-er-eq3}\\
\alpha\left(\frac{\dvg(\theta_{\mathbf{u}}^n-\theta_{\mathbf{u}}^{n-1})}{\Delta t},q_h\right)+(\dvg\,\theta_{\mathbf{w}}^n,q_h)&\,=\mathcal R_4^n(q_h)&&\forall q_h\in Q_h,\label{eq:Fu-er-eq4}
\end{align}
with
\begin{align}
  \mathcal R_1^n(\mathbf v_h):&=-2\mu(\varepsilon(\Pi_{\mathbf{u}}^n),\varepsilon(\mathbf{v}_h))+(\dvg\,\mathbf{v}_h,\Pi_{\psi}^n)+\alpha(\dvg\,\mathbf{v}_h,\Pi_{p}^n),\\
  \mathcal R_2^n(\zeta_h):&=\Upsilon_1+\Upsilon_2-\left(\frac{\dvg\,(\Pi_{\mathbf{u}}^n-\Pi_{\mathbf{u}}^{n-1})}{\Delta t},\zeta_h\right)
  -\frac{1}{\lambda}\left(\frac{\Pi_{\psi}^n-\Pi_{\psi}^{n-1}}{\Delta t},\zeta_h\right),\\
  \mathcal R_3^n(\mathbf r_h):&=-(\kappa^{-1}\Pi_{\mathbf{w}}^n,\mathbf{r}_h),\\
  \mathcal R_4^n(q_h):&=\Upsilon_3-\alpha\left(\frac{\dvg(\Pi_{\mathbf{u}}^n-\Pi_{\mathbf{u}}^{n-1})}{\Delta t},q_h\right)-(\dvg\,\Pi_{\mathbf{w}}^n,q_h).
\end{align}

We can now proceed with the error analysis for the fully discrete formulation, similar to the analysis we performed for the semi-discrete formulation, specifically for the case where $c_0 = 0$ and $\lambda \to \infty$. We note that this case is not covered in \cite{qi2021four}.
\begin{theorem}\label{thm:fullyerror}
  Let $c_0=0$, $\lambda\to\infty$, and let ($\mathbf{u},\psi, \mathbf{w}, p$),  ($\mathbf{u}_h^n,\psi_h^n, \mathbf{w}_h^n, p_h^n$) be the solutions of problems \eqref{continue:Biot1}-\eqref{continue:Biot4}
  and \eqref{fulldiscrete:Biot1}-\eqref{fulldiscrete:Biot4}, respectively. Then the following error estimate holds
  \begin{equation}\label{ineq:fullyerror}
  \begin{aligned}[b]
    &\max_{1\leq n\leq N}\|\mathbf u^n-\mathbf u_h^n\|_{1,\Omega}^2+\Delta t\sum_{n=1}^N\|\psi^n-\psi_h^n\|_{0,\Omega}^2\\
    +&\Delta t\sum_{n=1}^N\|\mathbf w^n-\mathbf w_h^n\|_{0,\Omega}^2+
    \Delta t\sum_{n=1}^N\|p-p_h^n\|_{0,\Omega}^2\leq C(h^{2\gamma}+(\Delta t)^2).
    \end{aligned}
  \end{equation}
\end{theorem}
\begin{proof}
We first sum equations \eqref{eq:Fu-er-eq1}-\eqref{eq:Fu-er-eq4},
and take $\mathbf{v}_h=\theta_{\mathbf{u}}^n-\theta_{\mathbf{u}}^{n-1},\,\zeta_h=\Delta t\,\theta_{\psi}^n,\,\mathbf{r}_h=\Delta t\,\theta_{\mathbf{w}}^n,\,q_h=\Delta t\,\theta_p^n$.
Using the inequalities
\begin{align}
\mu(\varepsilon(\theta_{\mathbf{u}}^n),\,\varepsilon(\theta_{\mathbf{u}}^n-\theta_{\mathbf{u}}^{n-1}))&\geq\frac{1}{2}\left(\|\varepsilon(\theta_{\mathbf{u}}^n)\|_0^2-\|\varepsilon(\theta_{\mathbf{u}}^{n-1})\|_0^2\right)\\
\left(\theta_{\psi}^n-\theta_{\psi}^{n-1},\,\theta_{\psi}^n\right) &\geq \frac{1}{2}\left(\|\theta_{\psi}^n\|_0^2-\|\theta_{\psi}^{n-1}\|_0^2\right),
\end{align}
 and the fact that  $\theta_{\mathbf{u}}^0=\mathbf0$, $\theta_{\psi}^0=0$, by summing on $n$ from 1 to $N$, we obtain the bound
\begin{align}\label{estimate:error}
\mu\,\|\varepsilon(\theta_{\mathbf{u}}^N)\|_0^2+\frac{1}{2\,\lambda}\|\theta_{\psi}^N\|_0^2+\,\Delta t \sum_{n=1}^N\|\kappa^{-\frac{1}{2}}\theta_{\mathbf{w}}^n\|_0^2\leq\sum_{n=1}^N \sum_{j=1}^4\mathcal{R}_j^n.
\end{align}
The  left-hand side of \eqref{estimate:error} can be relaxed by Korn's inequality, obtaining
\begin{align}\label{error:left}
\mu\|\varepsilon(\theta_{\mathbf{u}}^N)\|_0^2+\frac{1}{2\lambda}\|\theta_{\psi}^N\|_0^2+\Delta t \sum_{n=1}^N\|\kappa^{-\frac{1}{2}}\theta_{\mathbf{w}}^n\|_0^2
\geq& C_{Korn}\|\theta_{\mathbf{u}}^N\|_0^2+\frac{1}{2\lambda}\,\|\theta_{\psi}^N\|_0^2\notag\\&+\Delta t \sum_{n=1}^N\|\kappa^{-\frac{1}{2}}\theta_{\mathbf{w}}^n\|_0^2.
\end{align}
Therefore, our main task below is to derive a bound for the sum $\sum_{n=1}^N\mathcal{R}_j^n$, $j=1,\dots,4$. We will bound each term separately.

We first recall some useful tools which we will use: a discrete integration by parts formula applied to the grid functions and $f^n$ and $g^n$
\begin{align}
\sum_{n=1}^Nf^n(g^n-g^{n-1})=f^N-g^N-f^0g^0-\sum_{n=1}^N(f^n-f^{n-1})g^{n-1}\label{discrete integration ineq},
\end{align}
and the Taylor expansion
\begin{align}
\Pi_{\phi}^n-\Pi_{\phi}^{n-1}=\Delta t\,{\Pi_{\phi}^n}_t+\int_{t^{n-1}}^{t^n}\left(s-t^{n-1}\right){\Pi_{\phi}}_{tt}\,{\rm d}s,\,\phi=\mathbf{u},\,p\label{Taylor expansion}.
\end{align}
First, we bound the first term in $\sum_{n=1}^N\mathcal{R}_1^n$,
by using \eqref{discrete integration ineq}, \eqref{Taylor expansion}, Young's inequality, and the fact that $\theta_{\mathbf{u}}^n=0$:
\begin{align}\label{R1-1}
-2\mu&\sum_{n=1}^N(\varepsilon(\Pi_{\mathbf{u}}^n),\varepsilon(\theta_{\mathbf{u}}^n-\theta_{\mathbf{u}}^{n-1}))\notag\\
&=-2\mu\left(\varepsilon(\Pi_{\mathbf{u}}^N), \,\varepsilon(\theta_{\mathbf{u}}^N)\right)
+2\mu
\sum_{n=1}^N\left(\varepsilon(\Pi_{\mathbf{u}}^n)-\varepsilon(\Pi_{\mathbf{u}}^{n-1}),\,\varepsilon(\theta_{\mathbf{u}}^{n-1})\right)\notag\\
&=-2\mu\left(\varepsilon(\Pi_{\mathbf{u}}^N), \,\varepsilon(\theta_{\mathbf{u}}^N)\right)+2\mu\,\Delta t\sum_{n=1}^N\left(\varepsilon({\Pi_{\mathbf{u}_t}^n}),\,\varepsilon(\theta_{\mathbf{u}}^{n-1})\right)\notag\\
&+\,2\mu\,\sum_{n=1}^N\left(\int_{t^{n-1}}^{t^n}(s-t^{n-1})\varepsilon({\Pi_{\mathbf{u}_{tt}}}(s){\rm d}s,\, \varepsilon(\theta_{\mathbf{u}}^{n-1})\right)\notag\\
\leq &\epsilon_1\|\theta_{\mathbf{u}}^N\|_1^2+C\left(\|\Pi_{\mathbf{u}}^N\|_1^2+\Delta t\sum_{n=0}^N\left(\|{\Pi_{\mathbf{u}_t}^n}\|_1^2+\|\theta_{\mathbf{u}}^n\|_1^2 \right)+(\Delta t)^2\int_0^T\|{\Pi_{\mathbf{u}_{tt}}}(s)\|_1^2{\rm d}s\right).
\end{align}
By analogous computations, we can bound the remaining two terms of $\sum_{n=1}^N\mathcal{R}_1^n$:
\begin{align}
\label{R1-2}
\sum_{n=1}^N\left(\dvg\,(\theta_{\mathbf{u}}^n-\theta_{\mathbf{u}}^{n-1}),\,\Pi_{\psi}^n\right)
 =\left(\dvg\,(\theta_{\mathbf{u}}^N),\,\Pi_{\psi}^N \right)-\sum_{n=1}^N\left(\Pi_{\psi}^n-\Pi_{\psi}^{n-1},\dvg\,\theta_{\mathbf u}^{n-1}\right)\notag\\
\leq \epsilon_2\|\theta_{\mathbf{u}}^N\|_1^2+C\bigg(\|\Pi_{\psi}^N\|_0^2+\Delta t\sum_{n=0}^N\left(\|{\Pi_{\psi_t}^n}\|_0^2+\|\theta_{\mathbf{u}}^n\|_1^2 \right)
+(\Delta t)^2\int_0^T\|{\Pi_{\psi_{tt}}}(s)\|_0^2{\rm d}s\bigg),\\
\label{R1-3}
 \alpha\sum_{n=1}^N(\dvg\,(\theta_{\mathbf{u}}^n-\theta_{\mathbf{u}}^{n-1})),\Pi_{p}^n)
=\alpha\left(\dvg\,(\theta_{\mathbf{u}}^N),\,\Pi_p^N \right)-\alpha\,\sum_{n=1}^N\left(\Pi_p^n-\Pi_p^{n-1}, \dvg\,\theta_{\mathbf u}^{n-1}\right)\notag\\
 \leq \epsilon_3\|\theta_{\mathbf{u}}^N\|_1^2+C\bigg(\|\Pi_p^N\|_0^2+\Delta t\sum_{n=0}^N\left(\|{\Pi_{p_t}^n}\|_0^2+\|\theta_{\mathbf{u}}^n\|_1^2 \right)
 +(\Delta t)^2\int_0^T\|{\Pi_{p_{tt}}}(s)\|_0^2{\rm d}s\bigg).
\end{align}
Combining \eqref{R1-1}-\eqref{R1-3}, we have
\begin{align}
\sum_{n=1}^N\mathcal{R}_1^n&\leq(\epsilon_1+\epsilon_2+\epsilon_3)\|\theta_{\mathbf{u}}^N\|_1^2+ C\bigg(\|\Pi_{\mathbf{u}}^N\|_1^2+\|\Pi_{\psi}^N\|_0^2+\|\Pi_p^N\|_0^2+
\Delta t\sum_{n=0}^N\big(\|{\Pi_{\mathbf{u}_t}^n}\|_1^2\notag\\
&+\|{\Pi_{\psi_t}^n}\|_0^2+\|{\Pi_{p_t}^n}\|_0^2+\|\theta_{\mathbf{u}}^n\|_1^2 \big)
+(\Delta t)^2\int_0^T\|({\Pi_{\mathbf{u}_{tt}}}(s)\|_1^2\notag\\
&+\|{\Pi_{\psi_{tt}}}(s)\|_0^2+\|{\Pi_{p_{tt}}}(s)\|_0^2){\rm d}s\bigg).
\end{align}
Secondly, we estimate $\sum_{n=1}^N\Upsilon_3$ since it will be utilized repeatedly in the analysis below. We first recall the inequalities
\begin{align}
\left\|\int_{t^{n-1}}^{t^n}\left(s-t^{n-1}\right)\dvg\,\mathbf{u}_{tt}(s) \,{\rm d}s\right\|
\leq (\Delta t)^{\frac{3}{2}}\left(\int_{t^{n-1}}^{t^n}\|\dvg\,\mathbf{u}_{tt}(s)\|_0^2\,{\rm d}s\right)^{\frac{1}{2}},\label{ineq:u_{tt}}
\end{align}
\begin{align}
\left(\int_{t^{n-1}}^{t^n}\left(s-t^{n-1}\right)\dvg\,\mathbf{u}_{tt}(s) \,{\rm d}s, \theta_{p}^n\right)
\leq \left\|\int_{t^{n-1}}^{t^n}\left(s-t^{n-1}\right)\dvg\,\mathbf{u}_{tt}(s) \,{\rm d}s\right\|_0\left\|\theta_p^n\right\|_0.
\end{align}
Then, using \eqref{eq:Fu-er-eq3},\,\eqref{ineq:R-1},  we have
\begin{align}\label{use:lbb3}
\|\theta_p^n\|_0\leq C\left(\|\kappa^{-\frac{1}{2}}\,\theta_{\mathbf{w}}^n\|_0+\|\kappa^{-1}\,\Pi_{\mathbf{w}}^n\|_0\right).
\end{align}
From Young's inequality and \eqref{ineq:u_{tt}},\,\eqref{use:lbb3}, we obtain
\begin{align}\label{R4-1}
\sum_{n=1}^N\Upsilon_3\leq \frac{1}{4}\sum_{n=0}^N\Delta t\,\|\kappa^{-\frac{1}{2}}\,\theta_{\mathbf{w}}^n\|_0^2 + C\left(\Delta t\sum_{n=0}^N\|\Pi_{\mathbf{w}}^n\|_0^2+(\Delta t)^2\int_0^T\|\dvg\,\mathbf{u}_{tt}(s)\|_0^2\,{\rm d}s \right).
\end{align}
Then, the first two terms of $\sum_{n=1}^N\mathcal{R}_2^n$ can be bounded similarly by using \eqref{ineq:u_{tt}} and Young's inequality
\begin{align}\label{RR2_1}
\sum_{n=1}^N(\Upsilon_1+\Upsilon_2)
\leq C\left(\Delta t\sum_{n=0}^N\left\|\theta_{\psi}^n\right\|_0^2+(\Delta t)^2\int_0^T\Big(\|\dvg\,\mathbf{u}_{tt}(s)\|_0^2+\|\psi_{tt}(s)\|_0^2\Big)\,{\rm d}s\right).
\end{align}
\eqref{Taylor expansion} and Young's inequality yield
\begin{align}\label{RR2_2}
&\sum_{n=1}^N\bigg(-\left(\dvg\,\left(\Pi_{\mathbf{u}}^n-\Pi_{\mathbf{u}}^{n-1}\right),\theta_{\psi}^n\right)
-\frac{1}{\lambda}\left(\Pi_{\psi}^n-\Pi_{\psi}^{n-1},\,\theta_{\psi}^n\right)\bigg)\notag\\
\leq& C\bigg(\Delta t\sum_{n=1}^N\big(\,\|\theta_{\psi}^n\|_0^2+\,\|{\Pi_{\psi_t}^n}\|_0^2
+\|\Pi_{\mathbf{u}_t}^n\|_1^2\big)
+(\Delta t)^2\int_0^T(\|{\Pi_{\psi_{tt}}}(s)\|_0^2+\|\dvg\,\Pi_{\mathbf{u}_{tt}}(s)\|_0^2)\,{\rm d}s\bigg).
\end{align}
Combining \eqref{RR2_1}\,\eqref{RR2_2}, we have
\begin{align}
\sum_{n=1}^N\mathcal{R}_2^n&\leq C\bigg(\Delta t\sum_{n=1}^N\big(\,\|\theta_{\psi}^n\|_0^2+\,\|{\Pi_{\psi_t}^n}\|_0^2
+\|\Pi_{\mathbf{u}_t}^n\|_1^2\big)\notag\\
&+(\Delta t)^2\int_0^T(\|{\Pi_{\psi_{tt}}}(s)\|_0^2+\|\dvg\,\Pi_{\mathbf{u}_{tt}}(s)\|_0^2+\|\dvg\,\mathbf{u}_{tt}(s)\|_0^2+\|\psi_{tt}(s)\|_0^2)\,{\rm d}s\bigg).
\end{align}
We now bound $\sum_{n=1}^N\mathcal{R}_3^n$, using \eqref{use:lbb3}
\begin{align}
\sum_{n=1}^N\mathcal{R}_3^n=\sum_{n=1}^N\Delta t\,(\kappa^{-1}\Pi_{\mathbf{w}}^n,\theta_{\mathbf{w}}^n)
\leq
\frac{1}{4}\Delta t\sum_{n=1}^N\|\kappa^{-\frac{1}{2}}\theta_{\mathbf{w}}^n\|_0^2+C\Delta t\sum_{n=0}^N\|\dvg\,\Pi_{\mathbf{w}}^n\|_0^2.
\end{align}
The last term $\sum_{n=1}^N\mathcal{R}_4^n$, can be bounded by using \eqref{R4-1}\,\eqref{use:lbb3}\,\eqref{Taylor expansion}:
\begin{align}
\sum_{n=1}^N\mathcal{R}_4^n&\leq \frac{1}{2}\sum_{n=0}^N\Delta t\,\|\kappa^{-\frac{1}{2}}\,\theta_{\mathbf{w}}^n\|_0^2 + C \bigg(\Delta t\sum_{n=0}^N\big(\|\Pi_{\mathbf{w}}^n\|_0^2+\|\Pi_{\mathbf{u}_t}^n\|_1^2+ \|\Pi_{\mathbf{w}}^n\|_0^2+\|\dvg\,\Pi_{\mathbf{w}}^n\|_0^2\big)\notag\\
&+ (\Delta
t)^2\int_0^T\big(\|\dvg\,\mathbf{u}_{tt}(s)\|_0^2+\|\dvg\,\Pi_{\mathbf{u}_{tt}}(s)\|_0^2\big)\,{\rm d}s\bigg).
\end{align}
Consequently, combining the above bounds for $\sum_{n=1}^N\mathcal{R}_j^n,j=1,\dots,4$, and  \eqref{estimate:error},\,\eqref{error:left}, we obtain
\begin{align}
 (&C_{Korn}-\epsilon)\|\theta_{\mathbf{u}}^N\|_1^2+\frac{1}{2\,\lambda}\,\|\theta_{\psi}^N\|_0^2+\frac{1}{4}\,\Delta t \sum_{n=1}^N\|\kappa^{-\frac{1}{2}}\theta_{\mathbf{w}}^n\|_0^2\notag\\
 \leq& C\Bigg(\|\Pi_{\mathbf{u}}^N\|_1^2+\|\Pi_{\psi}^N\|_0^2+\|\Pi_p^N\|_0^2
 +\Delta t\sum_{n=0}^N\Big( \|\theta_{\mathbf{u}}^n\|_1^2 +\|{\Pi_{\mathbf{u}_t}^n}\|_1^2+\|{\Pi_{\psi_t}^n}\|_0^2+\|{\Pi_{p_t}^n}\|_0^2\notag\\
&+\|\theta_{\psi}^n\|_0^2+ \|\Pi_{\mathbf{w}}^n\|_0^2+\|\dvg\,\Pi_{\mathbf{w}}^n\|_0^2\Big)
+(\Delta t)^2\int_0^T\Big(\|{\Pi_{\mathbf{u}_{tt}}}(s)\|_1^2
+\|{\Pi_{\psi_{tt}}}(s)\|_0^2\notag\\&+\|{\Pi_{p_{tt}}}(s)\|_0^2
+\|\psi_{tt}(s)\|_0^2+\|\dvg\,\mathbf{u}_{tt}\|_0^2+\|\dvg\,\Pi_{\mathbf{u}_{tt}}\|_0^2\Big){\rm d}s \Bigg),\label{ineq:error1}
\end{align}
where $\epsilon=\epsilon_1+\epsilon_2+\epsilon_3$. We choose $\epsilon_1,\,\epsilon_2,\,\epsilon_3$ sufficiently small so that the constant $(C_{Korn}-\epsilon)$
on the left-hand side of the last inequality is positive. Notice that the assumption  $\lambda\to\infty$, implies that
$\frac{1}{2\,\lambda}\, \|\theta_{\psi}^n\|_0^2\to 0$, hence the term $\|\theta_{\psi}^n\|_0^2$ on the right-hand side cannot be bounded by the left-hand side of the inequality.
Therefore, we need to make additional estimates of the right-hand side term $\|\theta_{\psi}^n\|_0^2$.

Similarly to \eqref{use:lbb3}, \eqref{eq:Fu-er-eq1}, we apply Lemma \ref{lem:THLBB} to find
\begin{align}
\|\theta_{\psi}^n\|_0^2
=&2\mu(\varepsilon(\theta_{\mathbf{u}}^n)+\varepsilon(\Pi_{\mathbf{u}}^n),\varepsilon(\mathbf{v}_h))-\alpha(\dvg\,\mathbf{v}_h,\theta_{p}^n)
-(\dvg\,\mathbf{v}_h,\Pi_{\psi}^n)-\alpha(\dvg\,\mathbf{v}_h,\Pi_{p}^n)\notag\\
\leq& C\Big(\|\varepsilon(\theta_{\mathbf{u}}^n)\|_0+\|\varepsilon(\Pi_{\mathbf{u}}^n)\|_0+\|\theta_{p}^n\|_0+\|\Pi_{\psi}^n\|_0+\|\Pi_{p}^n\|_0\Big)\|\mathbf{v}_h\|_1\notag\\
\leq &
C\Big(\|\theta_{\mathbf{u}}^n\|_1+\|\Pi_{\mathbf{u}}^n\|_1+\|\kappa^{-\frac{1}{2}}\,\theta_{\mathbf{w}}^n\|_0+\|\,\Pi_{\mathbf{w}}^n\|_0+\|\Pi_{\psi}^n\|_0+\|\Pi_{p}^n\|_0\Big)\|\theta_{\psi}^n\|_0,
\end{align}
that is
\begin{equation}\label{use:lbb2}
  \|\theta_{\psi}^n\|_0\leq C\Big(\|\theta_{\mathbf{u}}^n\|_1+\|\Pi_{\mathbf{u}}^n\|_1+\|\kappa^{-\frac{1}{2}}\,\theta_{\mathbf{w}}^n\|_0+\|\,\Pi_{\mathbf{w}}^n\|_0+\|\Pi_{\psi}^n\|_0+\|\Pi_{p}^n\|_0\Big).
\end{equation}
Now, we can rewrite \eqref{ineq:error1} to obtain
\begin{align}
&\|\theta_{\mathbf{u}}^N\|_1^2+\,\Delta t \sum_{n=1}^N\|\kappa^{-\frac{1}{2}}\theta_{\mathbf{w}}^n\|_0^2\notag\\
 \leq& C\Bigg(\|\Pi_{\mathbf{u}}^N\|_1^2+\|\Pi_{\psi}^N\|_0^2+\|\Pi_p^N\|_0^2
 +\Delta t\sum_{n=0}^N\Big( \|\theta_{\mathbf{u}}^n\|_1^2 +\|{\Pi_{\mathbf{u}_t}^n}\|_1^2+\|{\Pi_{\psi_t}^n}\|_0^2+\|{\Pi_{p_t}^n}\|_0^2\notag\\&
+\|\Pi_{\mathbf{u}}^n\|_1^2+\|\Pi_{\psi}^n\|_0^2+\|\Pi_{p}^n\|_0^2+ \|\Pi_{\mathbf{w}}^n\|_0^2+\|\dvg\,\Pi_{\mathbf{w}}^n\|_0^2\Big)
+(\Delta t)^2\int_0^T\Big(\|{\Pi_{\mathbf{u}_{tt}}}(s)\|_1^2\notag\\&
+\|{\Pi_{\psi_{tt}}}(s)\|_0^2
+\|{\Pi_{p_{tt}}}(s)\|_0^2
+\|\psi_{tt}(s)\|_0^2+\|\dvg\,\mathbf{u}_{tt}\|_0^2+\|\dvg\,\Pi_{\mathbf{u}_{tt}}\|_0^2\Big){\rm d}s \Bigg).
\end{align}
The discrete Gronwall inequality (see \cite{HR1990}) and Lemma \ref{lem:app} imply that
\begin{equation}\label{ineq:fullyR-1}
\begin{aligned}
 &\max_{1\leq n\leq N} \|\theta_{\mathbf{u}}^n\|_{1}^2+\Delta t\sum_{n=1}^N \|\theta_{\mathbf{w}}^n\|_0^2\\
 \leq&C\Bigg( h^{2\gamma}\max_{1\leq n\leq N}\Big(\|\mathbf u^n\|_{\gamma+1}^2
+\|p^n\|_{\gamma}^2+\|\psi^n\|_{\gamma}^2\Big)+h^{2\gamma}(\Delta t)\max_{1\leq n\leq N}\Big(\|\mathbf u_t^n\|_{\gamma+1,\Omega}^2+\|\psi_t^n\|_{\gamma}^2\\
 &+\|p_t^n\|_{\gamma}^2+\|\mathbf w^n\|_{\mathbf H^\gamma(\dvg;\Omega)}^2 \Big)+(\Delta t)^2\int_0^T\Big(\|{\Pi_{\mathbf{u}_{tt}}}(s)\|_1^2
+\|{\Pi_{\psi_{tt}}}(s)\|_0^2+\|{\Pi_{p_{tt}}}(s)\|_0^2\\&
+\|{\Pi_{\psi_{tt}}}(s)\|_0^2+\|\psi_{tt}(s)\|_0^2+\|\dvg\,\mathbf{u}_{tt}\|_0^2+\|\dvg\,\Pi_{\mathbf{u}_{tt}}\|_0^2\Big){\rm d}s\Bigg).
 \end{aligned}
\end{equation}
Finally, using \eqref{use:lbb3} and \eqref{use:lbb2}, we estimate the terms with  $p$ and $\psi$, completing the proof.
\end{proof}

Similarly, when $c_0\geq\tau_0>0$, $\lambda_1\leq \tau_1<\infty$, we have the following estimate.
\begin{theorem}\label{thm:fullyerror0}
  Let $c_0\geq\tau_0>0$, $\lambda_1\leq \tau_1<\infty$, and let ($\mathbf{u},\psi, \mathbf{w}, p$) and ($\mathbf{u}_h^n,\psi_h^n, \mathbf{w}_h^n, p_h^n$) be the solutions of problems \eqref{continue:Biot1}-\eqref{continue:Biot4}
  and \eqref{fulldiscrete:Biot1}-\eqref{fulldiscrete:Biot4}, respectively. The the following error estimate holds
  \begin{equation}\label{ineq:fullyerror0}
  \begin{aligned}[b]
    &\max_{1\leq n\leq N}\|\mathbf u^n-\mathbf u_h^n\|_1^2+\max_{1\leq n\leq N}\|\psi^n-\psi_h^n\|_0^2\\
    +&\Delta t\sum_{n=1}^N\|\mathbf w^n-\mathbf w_h^n\|_0^2+
    \max_{1\leq n\leq N}\|p^n-p_h^n\|_0^2\leq C(h^{2\gamma}+(\Delta t)^2).
    \end{aligned}
  \end{equation}
\end{theorem}

\section{A further discussion on time discretization.}\label{sec:time}
In the previous sections, we employed the backward Euler method to discretize the time derivative in order to facilitate the analysis and derive the fully discrete scheme. It is noted that the error order achieved through isogeometric analysis for spatial discretization is $k + 1$. To achieve a more balanced error order between the time and spatial discretization in the fully discrete problem, we can utilize higher-order schemes for time discretization, in contrast to the backward Euler method.

The backward Euler method for the time discretization formulation proposed previously can be easily extended to the Crank–Nicolson method, which is an implicit second-order method in the time direction. Specifically, we have the following discrete problem:
\begin{align}
2\mu(\varepsilon(\mathbf{u}_h^n),\varepsilon(\mathbf{v}_h))-(\dvg\,\mathbf{v}_h,\psi_h^n)-\alpha(\dvg\,\mathbf{v}_h,p_h^n)
\,&=(\mathbf{f}_{\mathbf{u}}^n,\mathbf{v}_h),\label{fulldiscrete:Biot1CN}\\
 \left(\frac{\dvg\,\mathbf{u}_h^n-\dvg\,\mathbf{u}_h^{n-1}}{\Delta t},\zeta_h\right)+\frac{1}{\lambda}\left(\frac{\psi_h^n-\psi_h^{n-1}}{\Delta t},\zeta_h\right)\,&=0,\label{fulldiscrete:Biot2CN}\\
(\kappa^{-1}\mathbf{w}_h^n,\mathbf{r}_h)-(\dvg\,\mathbf{r}_h,p_h^n)\,&=0,\label{fulldiscrete:Biot3CN}\\
c_0\left(\frac{p_h^n-p_h^{n-1}}{\Delta t},q_h\right)+\alpha\left(\frac{\dvg\,\mathbf{u}_h^n-\dvg\,\mathbf{u}_h^{n-1}}{\Delta t},q_h\right)&\notag\\+\left(\frac{\dvg\,\mathbf{w}_h^n+\dvg\,\mathbf w_h^{n-1}}{2},q_h\right)\,&=\left(\frac{f_p^n+f_p^{n-1}}{2},q_h\right).\label{fulldiscrete:Biot4CN}
\end{align}
Compared with equations \eqref{derive:Biot1}-\eqref{derive:Biot4}, only the last equation has been improved.

Both the backward Euler method and the Crank–Nicolson method are linear single-step methods. We can still continue to explore multi-step methods, such as the well-known Backward Differentiation Formula (BDF). For the derivative in the time direction, we use the following $k$-th order BDF scheme for discretization:
\[\dvg\mathbf u_t\thickapprox\frac1{\Delta t}\sum_{j=0}^k\lambda_j^k\dvg\mathbf u_h^{n-j},\]
where $\mathbf u_h^{n-j}$ are
are approximate solutions at $t_{n-j}$ and $\lambda_j^k$  are given, see e.g., \cite{HS2007,liujie}. We propose the following scheme
\begin{align}
2\mu(\varepsilon(\mathbf{u}_h^n),\varepsilon(\mathbf{v}_h))-(\dvg\,\mathbf{v}_h,\psi_h^n)-\alpha(\dvg\,\mathbf{v}_h,p_h^n)
\,&=(\mathbf{f}_{\mathbf{u}}^n,\mathbf{v}_h),\label{fulldiscrete:Biot1BDF}\\
 \Big(\frac1{\Delta t}\sum_{j=0}^k\lambda_j^k\dvg\mathbf u_h^{n-j},\zeta_h\Big)+\frac{1}{\lambda}\Big(\frac1{\Delta t}\sum_{j=0}^k\lambda_j^k \psi_h^{n-j},\zeta_h\Big)\,&=0,\label{fulldiscrete:Biot2BDF}\\
(\kappa^{-1}\mathbf{w}_h^n,\mathbf{r}_h)-(\dvg\,\mathbf{r}_h,p_h^n)\,&=0,\label{fulldiscrete:Biot3BDF}\\
c_0\Big(\frac1{\Delta t}\sum_{j=0}^k\lambda_j^kp_h^{n-j},q_h\Big)+\alpha\Big(\frac1{\Delta t}\sum_{j=0}^k\lambda_j^k\dvg\mathbf u_h^{n-j},q_h\Big)&\notag\\+(\dvg \mathbf w_h^n,q_h)\,&=(f_p^n,q_h)\label{fulldiscrete:Biot4BDF}.
\end{align}
The convergence properties of formulations will be presented numerically in the next section.

\section{Numerical experiments}\label{section5}
In this section, we present several numerical experiments in order to validate the theoretical estimates proved in the previous sections, as well as the accuracy and efficiency of the proposed isogeometric methods for the Biot model. We consider 6 tests with Biot model given on 2D and 3D domains and over the time interval $[0,1]$.
Tests 1 and 2 are conducted to assess the order of convergence of our isogeometric discretization.
Test 3  focuses instead on how the accuracy of our method depends on the Lam\'{e} constant $\lambda$ and the storage coefficient $c_0$. In Test 4, we consider two cantilever bracket models in two and three dimensions in order to show that our algorithm does not lead to pressure oscillations.
In  Test 5, we examine the convergence of the schemes proposed in Section \ref{sec:time} with respect to time discretization.
In the last test, we show that our isogeometric analysis method exhibits superior convergence accuracy compared to the standard finite element method.

We define the following errors for each of the four fields in our Biot model:
\begin{align*}
&\mathbf{E}_{\mathbf u}=\max_{1\leq n\leq N}\|\mathbf u^n-\mathbf u_h^n\|_1, &&
 \mathbf{E}_{\psi}=
\begin{cases}
\max_{1\leq n\leq N}\|\psi^n-\psi_h^n\|_0&\lambda<\infty\\
\Delta t\sum_{n=1}^N\|\psi^n-\psi_h^n\|_0 & \mbox{otherwise},\
\end{cases}\\
&\mathbf{E}_{\mathbf w}=\Delta t\sum_{n=1}^N\|\mathbf w^n-\mathbf w_h^n\|_0,
&&
\mathbf{E}_p=
\begin{cases}
\Delta t\sum_{n=1}^N\|p^n-p_h^n\|_0&c_0=0\\
 \max_{1\leq n\leq N}\|p^n-p_h^n\|_0& \mbox{otherwise}.
\end{cases}
\end{align*}

In all our tests, we choose $k_{\mathbf v}=k_p,\,p_{\mathbf v}=p_p+1$ and therefore $\gamma = p_p+1$, according to the definition of $\gamma$ in Theorem \ref{thm:semierror0}. Moreover, we require that the spline regularity indices $k_{\mathbf v}$ and $k_p$ satisfy the constraint \eqref{cond:lbb}.
Since Theorem \ref{thm:fullyerror} and Theorem \ref{thm:fullyerror0} show that the rate of convergence is controlled by both the mesh size $h$ and the time-step size $\Delta t$, we have chosen to assess the rate of convergence in our tests by reducing the time-step size $\Delta t$ proportionally to the mesh size $h$, so that the ratio $h^\gamma/\Delta t = constant$ determined by the initial mesh size $h_0=1/6$ and the initial time-step size $\Delta t=1$.

$\mathbf{Test\,1.}$
In this numerical test, we consider Biot's poroelastic equations in a quarter annulus domain $\Omega$ defined in polar coordinates by $(\rho,\varphi)\in(2,4)\times(0,\frac\pi2)$. The source terms $\mathbf{f_u}$ and $f_p$ are selected so that the exact solution $(\mathbf{u},\,\psi,\mathbf{w},\,p)$ is
\begin{align*}
&u_1(x,y,t)=u_2(x,y,t)=\sin(\pi x)\,\sin(\pi y)\,e^{t},
&\psi(x,y,t)=-\lambda\dvg\,\mathbf{u},\hspace{2.15cm}\\
&\mathbf{w}(x,y,t)=-\kappa\nabla p,
&p(x,y,t)=\,(\cos(\pi x)+0.1505)\,e^{-t}.
\end{align*}

\begin{table}[tb!]
\centering
\footnotesize
\caption{Convergence errors and orders for Test 1
on the annulus domain}
\label{table1}
\begin{tabular}{r|cccccccc}
\hline
 $1/h$(dof) & $\mathbf E_{\mathbf u}$ & order& $\mathbf E_p$ & order & $\mathbf E_{\mathbf w}$ & order& $\mathbf E_{\psi}$ & order\\ \hline 
 \multicolumn{9}{c}{$p_{\mathbf v}=2,\,p_p=1$}\\
  6(425) &1.13e+00& &3.12e-01& &9.66e-01&  &3.76e-01& \\
 12(1079) &2.73e-01&2.04&9.81e-02&1.67&2.05e-01&2.24&8.97e-02&2.07\\
 18(3857) &1.21e-01&2.00&4.54e-02&1.90&8.77e-02&2.09&3.92e-02&2.04\\
 24(6869) &6.82e-02&2.00&2.59e-02&1.95&4.87e-02&2.04&2.19e-02&2.02\\
 30(10745) &4.36e-02&2.00&1.67e-02&1.97&3.10e-02&2.03&1.40e-02&2.01\\
\hline
\multicolumn{9}{c}{$p_{\mathbf v}=3,\,p_p=2$}\\ 
  6(1229) &1.91e-01& &3.10e-01& &9.61e-01&  &1.04e-01& \\
 12(4901) &3.02e-02&2.66&5.09e-02&2.61&9.85e-02&3.29&1.67e-02&2.63\\
 18(11021) &9.07e-03&2.97&1.55e-02&2.94&2.85e-02&3.06&5.13e-03&2.91\\
 24 (19589)&3.83e-03&2.99&6.57e-03&2.98&1.20e-02&3.02&2.19e-03&2.96\\
 30 (30605)&1.96e-03&3.00&3.37e-03&2.99&6.12e-03&3.01&1.12e-03&3.00\\
\hline
\multicolumn{9}{c}{$p_{\mathbf v}=4,\,p_p=3$}\\

 6(2465) &1.15e-01& &3.10e-01& &9.61e-01& &9.12e-02&\\
 12(9821) &9.37e-03&3.62&2.59e-02&3.58&4.84e-02&4.31&7.59e-03&3.59\\
 18(22073) &1.88e-03&3.97&5.19e-03&3.96&9.44e-03&4.03&1.52e-03&3.96\\
 24(39221) &5.95e-04&3.99&1.65e-03&3.99&2.98e-03&4.01&4.82e-04&3.99\\
 30(61265) &2.44e-04&4.00& 6.76e-04&4.00&1.22e-03 &4.00&1.98e-04&3.99\\

\hline
\multicolumn{9}{c}{$p_{\mathbf v}=5,\,p_p=4$}\\

 6(4133) &1.10e-01& &3.10e-01& &9.61e-01& &9.10e-02& \\
12(16469) &4.64e-03&4.57&1.31e-02&4.57&2.40e-02&5.32&3.83e-03&4.57\\
 18(37013) &6.16e-04&4.98&1.74e-03&4.98&3.14e-03&5.02&5.08e-04&4.98\\
 24(65765)&1.46e-04& 5.00&4.14e-04&4.99& 7.45e-04&5.00&1.21e-04&4.99\\
 30(102725)&4.78e-05&5.00&1.36e-04&5.00& 2.44e-04&5.00&3.98e-05&4.99\\

 \hline
\multicolumn{9}{c}{$p_{\mathbf v}=6,\,p_p=5$}\\
 6(6233)&1.10e-01&  &3.10e-01& &9.61e-01& &9.10e-02&  \\
 12(24845) &2.33e-03&5.56&6.57e-03&5.56&1.20e-02&6.33&1.92e-03&5.56\\
18 (55841) &2.05e-04&5.99&5.79e-04&5.99& 1.05e-03  &6.01&1.70e-04&5.98\\
24 (99221) &3.64e-05&6.00& 1.03e-04&6.00& 1.87e-04&6.00&3.03e-05&5.99\\
\hline
\end{tabular}
\end{table}
\begin{figure}[tb!]
\centering
\includegraphics[trim={3cm 0 1cm 0},clip,width=1.1\textwidth]{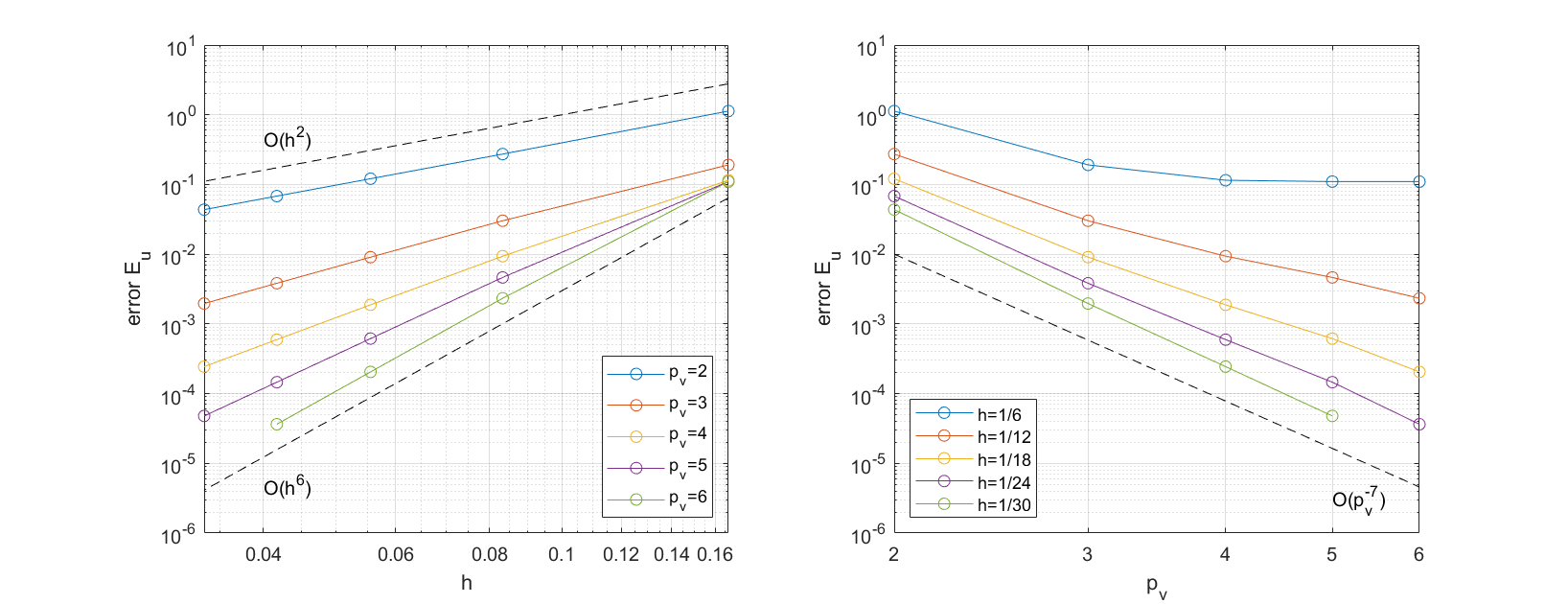}
\caption{Displacement errors $\mathbf{E}_{\mathbf u}$ from Table 1 as a function of: (left) the mesh size $h$ for fixed $p_{\mathbf v}$; (right) the polynomial degree $p_{\mathbf v}$ for fixed $h$}
\label{plot_table1_h_p}
\end{figure}

We choose the material parameters   $c_0=\kappa=\lambda=\mu=1$.
The boundary conditions for $\mathbf{u}$ and $p$ are determined by the exact solution.
We employ the isogeometric Taylor-Hood elements proposed in \cite{BS2013}.
Table \ref{table1},
reports the convergence errors and orders for each of the four solution fields $(\mathbf{u},\,\psi,\mathbf{w},\,p)$ when the mesh size $h$ is refined, for each choice of spline polynomial degree ranging from $p_{\mathbf v}=2,\,p_p=1$ to $p_{\mathbf v}=6,\,p_p=5$.
The results show the optimal convergence predicted by our theoretical estimates. Even with high-order splines, our isogeometric discretization consistently attains the optimal order $p_{\mathbf v}$ for each solution field when the mesh size $h$ is refined.

In Figure \ref{plot_table1_h_p}, we show the semilogarithmic plots of the displacement errors $\mathbf{E}_{\mathbf u}$ from Table 1 as a function of the mesh size $h$ for fixed spline polynomial degree $p_{\mathbf v}$ (left panel) and
as a function of the $p_{\mathbf v} = 2,\dots, 6$, for fixed $h$ (right panel). The $O(h^{p_{\mathbf v}}$) convergence order in $h$ is clearly visible in the left panel, while the results in the right panel show that the convergence plots in $p_{\mathbf v}$ associated with smaller values of $h$ are steeper
(we plotted for example the scaled power curve $O(p_{\mathbf v}^{-7})$ (dashed line). Additional tests with higher degrees $p_{\mathbf v}$ would be needed to asses whether the convergence is exponential or not.

$\mathbf{Test\,2.}$ In this test, we consider an L-shape domain, specifically $\Omega=(-1,1)^2\setminus(0,1)^2$, and a Biot problem with the following exact solution:
\begin{align*}
&\mathbf u(x,y,t)=e^t
 \begin{pmatrix}
   x(1-x)y(1-y)\\
   \sin(\pi x)\,\sin(\pi y)
\end{pmatrix},
&\psi(x,y,t)=-\lambda\dvg\,\mathbf{u},
\hspace{1.6cm}\\
&\mathbf{w}(x,y,t)=-\kappa\nabla p,
&p(x,y,t)=e^{-t}\cos(\pi x)\,\sin(\pi y).
\end{align*}
Figure \ref{Lshapefig} shows the results obtained with different isogeometric Taylor-Hood pairs with spline degree $p_{\mathbf v}$ increasing from 2 (first row) to 5 (fourth row) and different spline regularity $k_{\mathbf v}$ ranging from the minimal value 0 to the maximal value $p_{\mathbf v}-2$. The results clearly show that our method achieves the optimal order predicted by the  theoretical estimates, for all spline regularities.

\begin{figure}[htb!]
\begin{tabular}{cccc}
\centering
\subfigure [$p_{\mathbf v}=2,k_{\mathbf v}=0$]
{\begin{minipage}[t]{4cm}
\centering
\includegraphics[trim={0.2cm 0 0 0},clip,width=3.8cm]{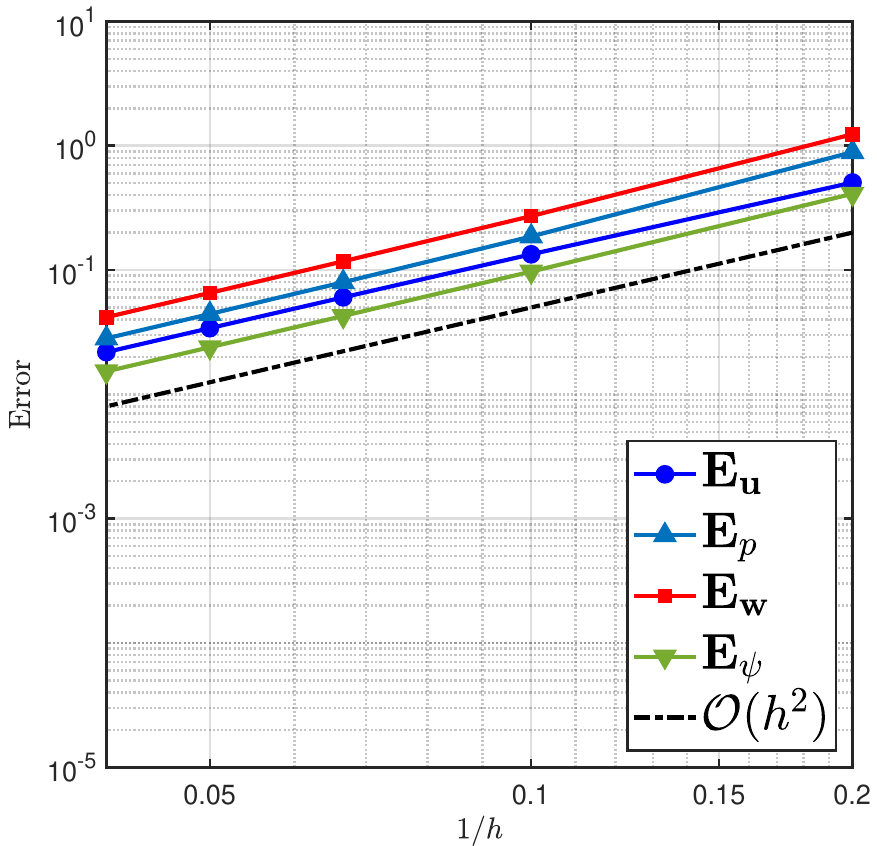}
\end{minipage}
} & & & \\
\hline
\subfigure [$p_{\mathbf v}=3,k_{\mathbf v}=0$]
{\begin{minipage}[t]{4cm}
\centering
\includegraphics[trim={0.2cm 0 0 0},clip,width=3.8cm]{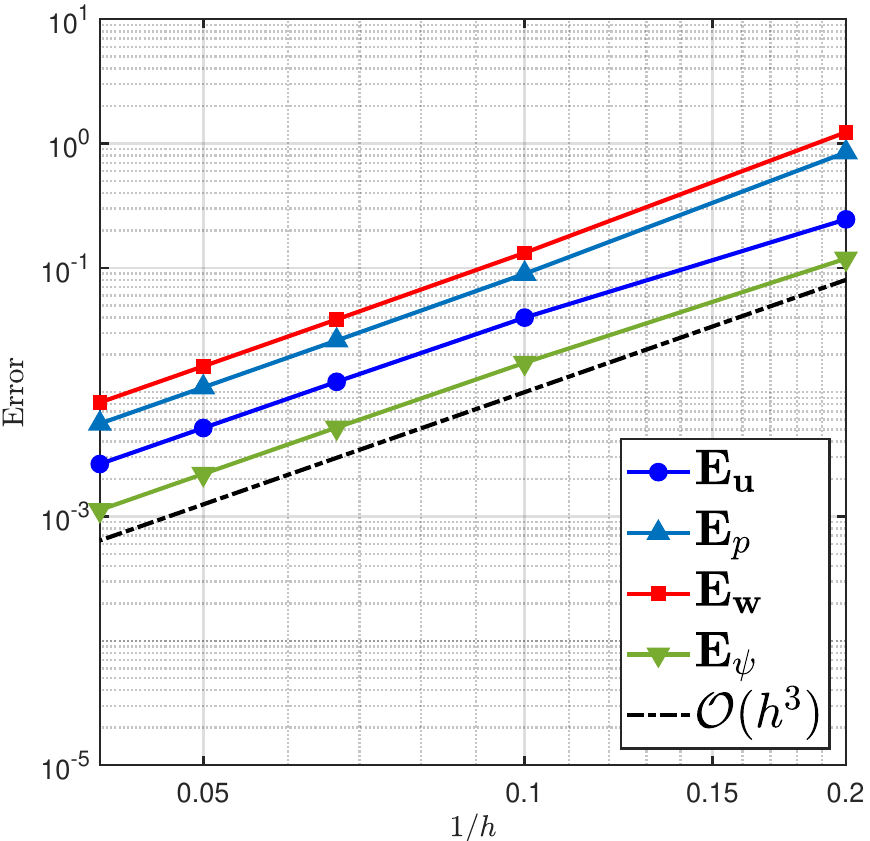}
\end{minipage}
} &
\hspace{-6mm}\subfigure [$p_{\mathbf v}=3,k_{\mathbf v}=1$]
{\begin{minipage}[t]{4cm}
\centering
\includegraphics[trim={0.5cm 0 0 0},clip,width=3.7cm]{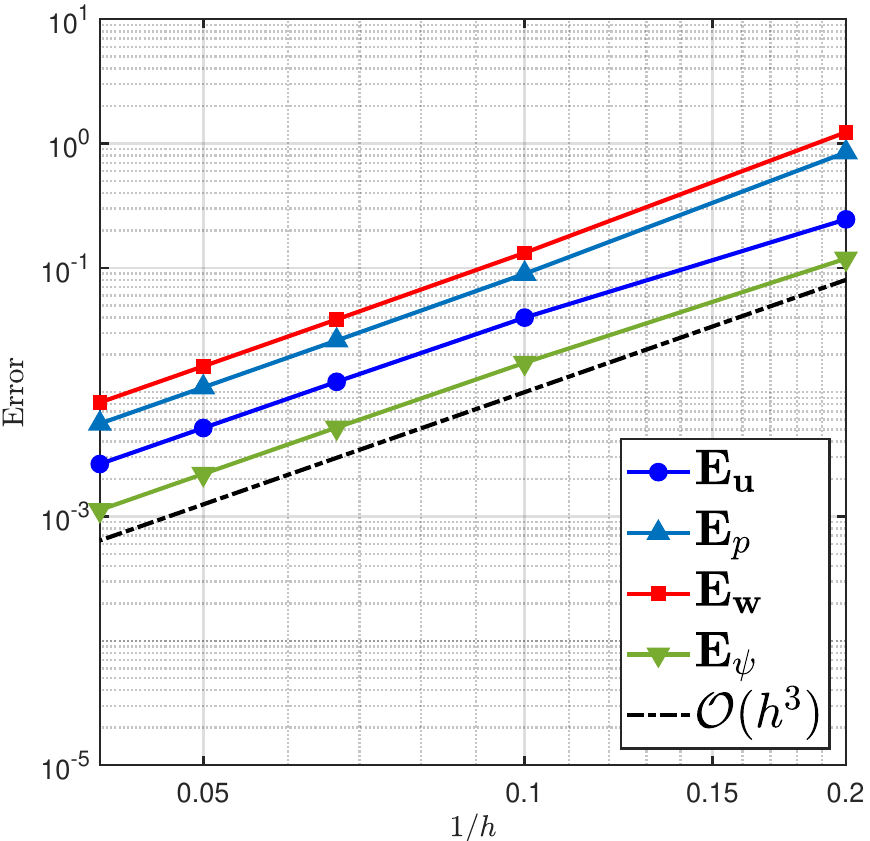}
\end{minipage}
} & & \\
\hline
\centering
\subfigure [$p_{\mathbf v}=4,k_{\mathbf v}=0$]
{\begin{minipage}[t]{4cm}
\centering
\includegraphics[trim={0.2cm 0 0 0},clip,width=3.8cm]{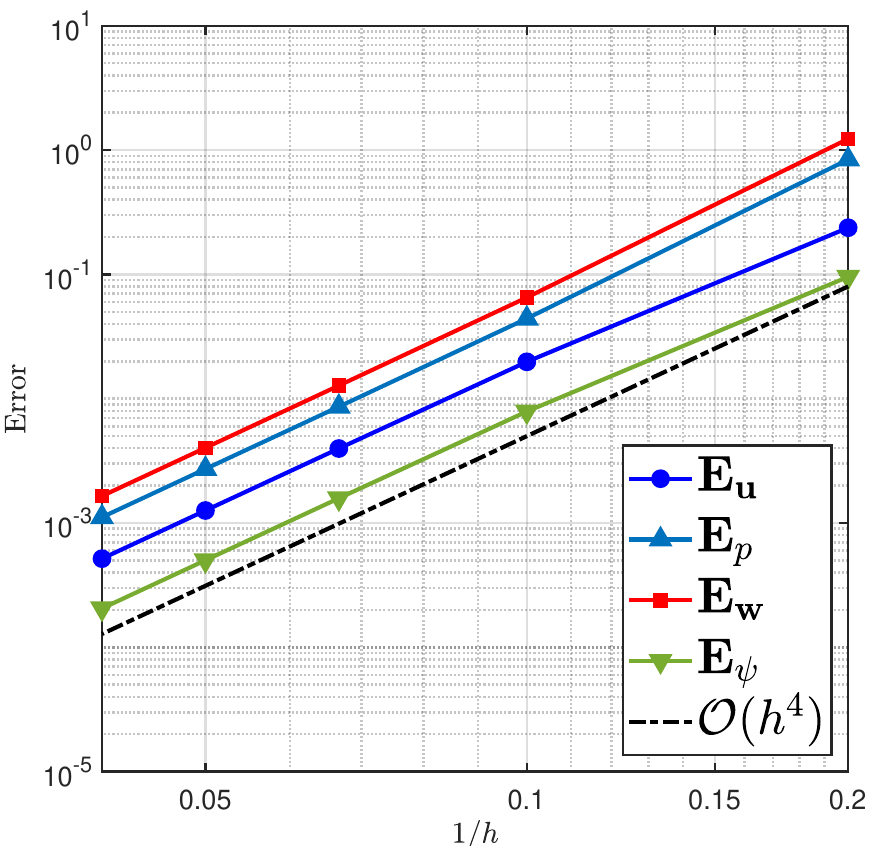}
\end{minipage}
} &
\hspace{-6mm}\subfigure [$p_{\mathbf v}=4,k_{\mathbf v}=1$]
{\begin{minipage}[t]{4cm}
\centering
\includegraphics[trim={0.5cm 0 0 0},clip,width=3.7cm]{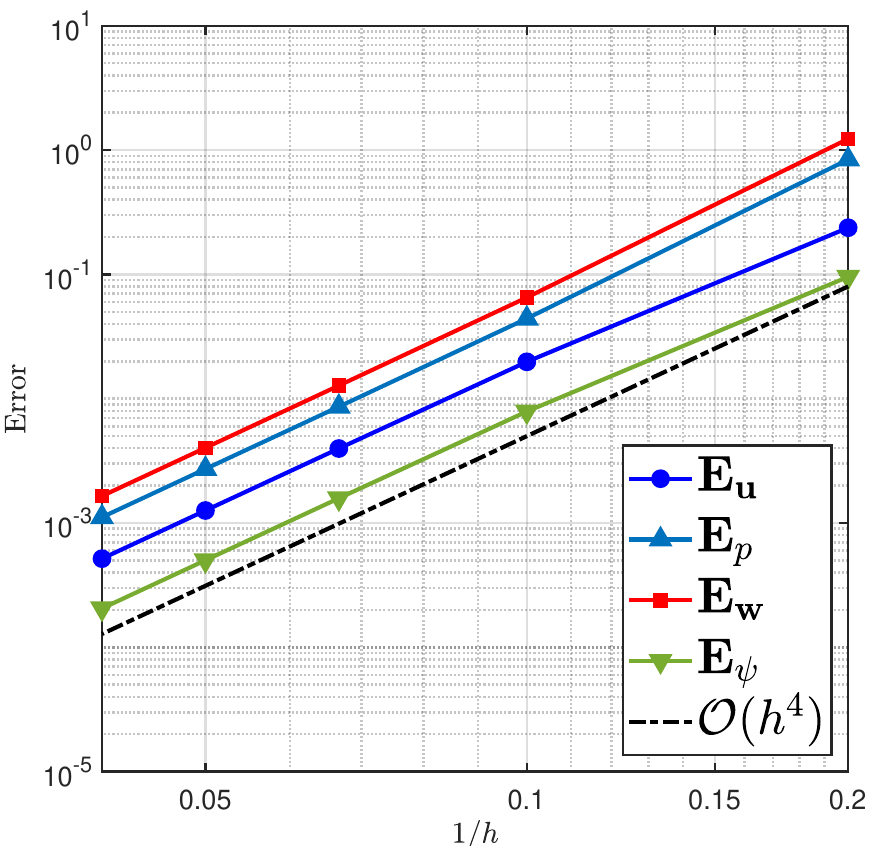}
\end{minipage}
} &
\hspace{-9mm}\subfigure [$p_{\mathbf v}=4,k_{\mathbf v}=2$]
{\begin{minipage}[t]{4cm}
\centering
\includegraphics[trim={0.5cm 0 0 0},clip,width=3.7cm]{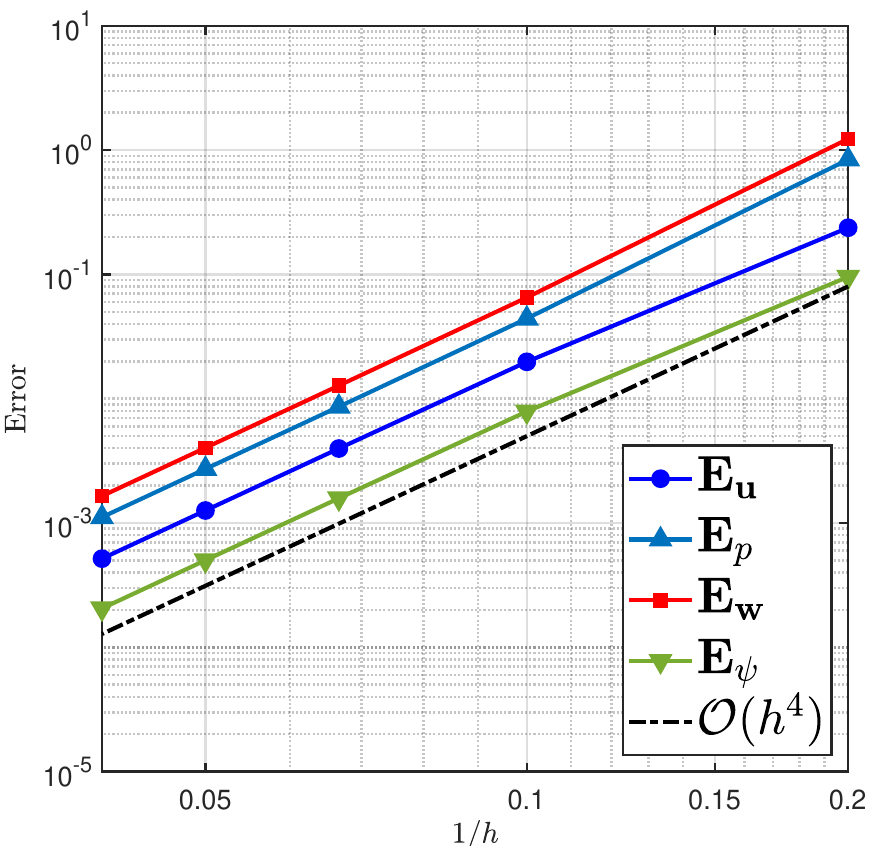}
\end{minipage}
} &  \\
\hline
\subfigure [$p_{\mathbf v}=5,k_{\mathbf v}=0$]
{\begin{minipage}[t]{4cm}
\centering
\includegraphics[trim={0.2cm 0 0 0},clip,width=3.8cm]{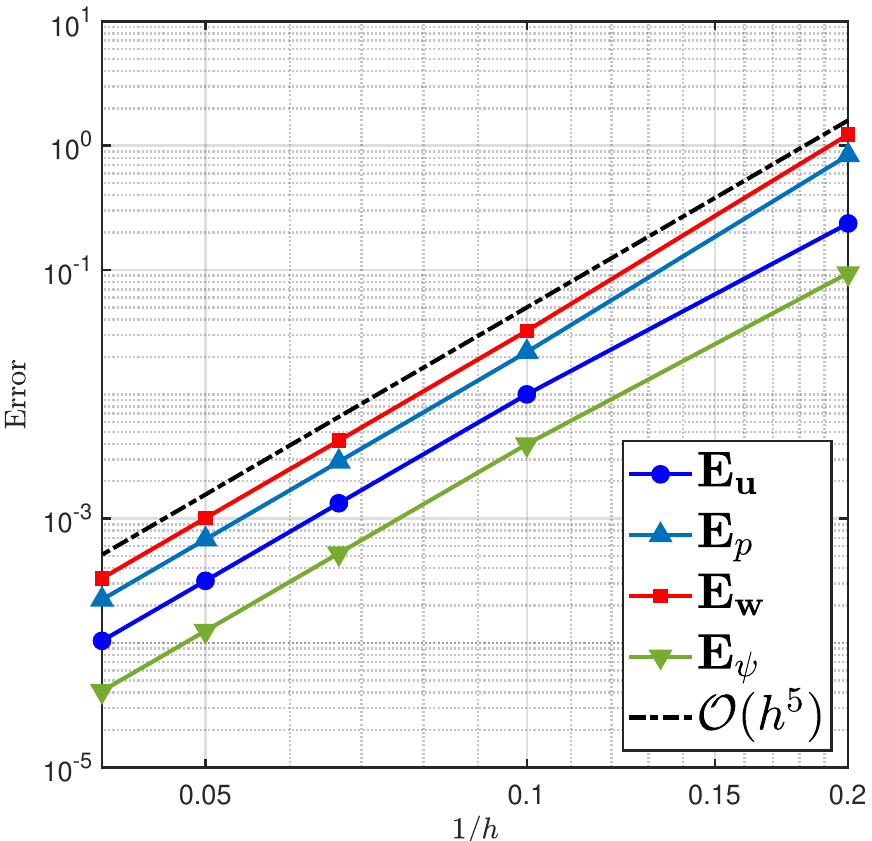}
\end{minipage}
} &
\hspace{-6mm}\subfigure [$p_{\mathbf v}=5,k_{\mathbf v}=1$]
{\begin{minipage}[t]{4cm}
\centering
\includegraphics[trim={0.5cm 0 0 0},clip,width=3.7cm]{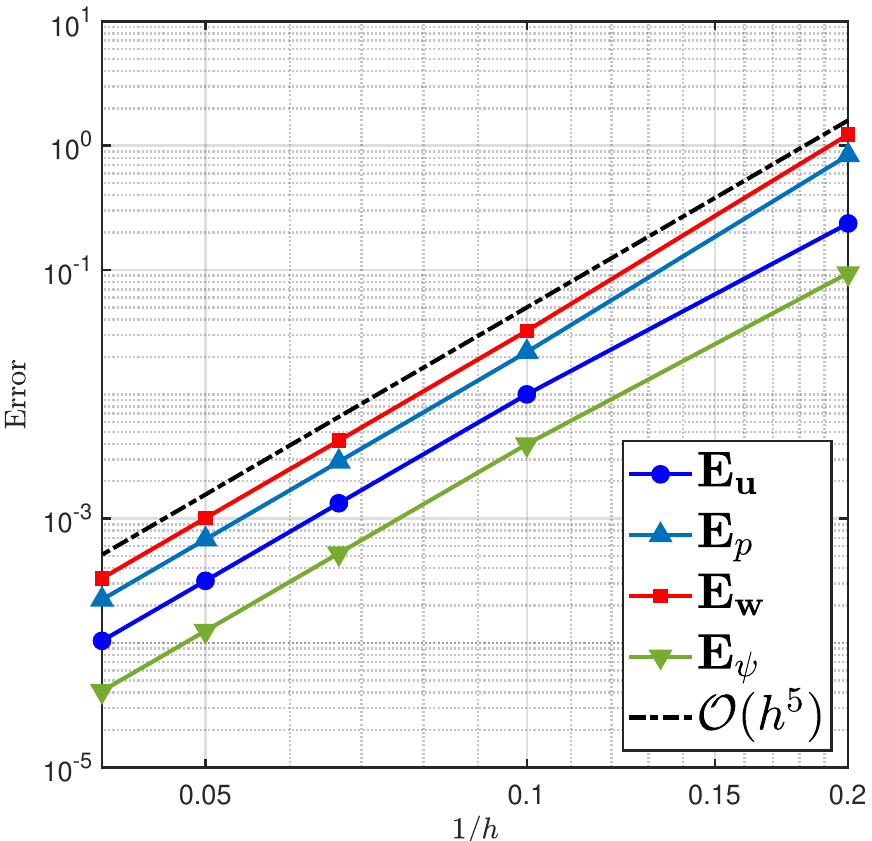}
\end{minipage}
} &
\hspace{-9mm}\subfigure [$p_{\mathbf v}=5,k_{\mathbf v}=2$]
{\begin{minipage}[t]{4cm}
\centering
\includegraphics[trim={0.5cm 0 0 0},clip,width=3.7cm]{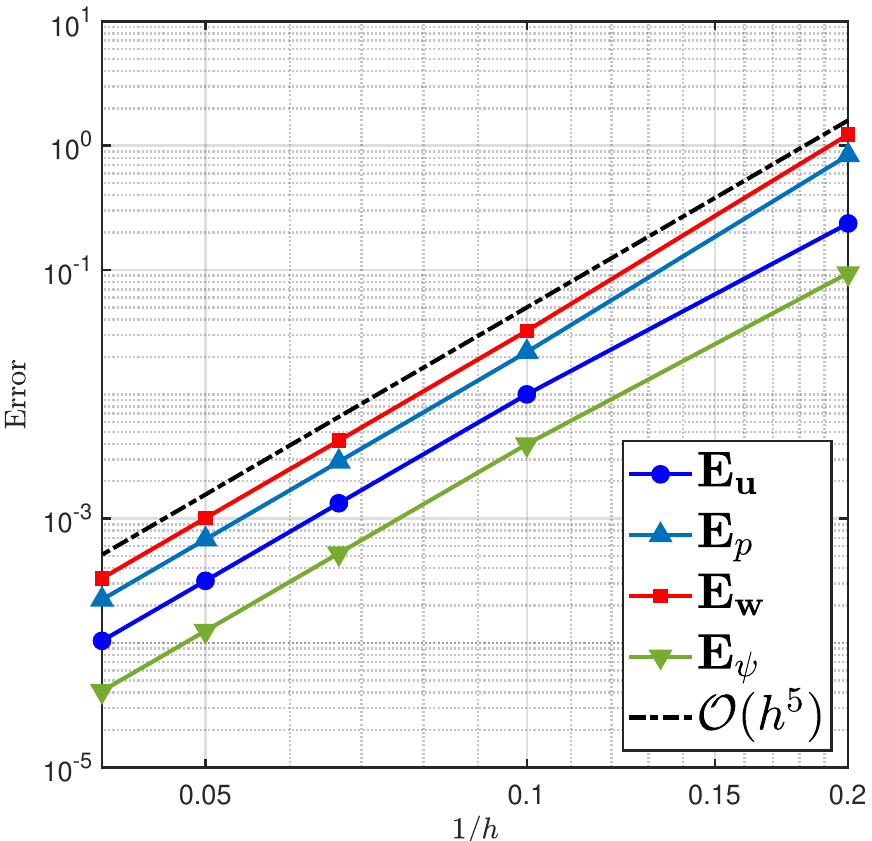}
\end{minipage}
} &
\hspace{-9mm}\subfigure [$p_{\mathbf v}=5,k_{\mathbf v}=3$]
{\begin{minipage}[t]{4cm}
\centering
\includegraphics[trim={0.5cm 0 0 0},clip,width=3.7cm]{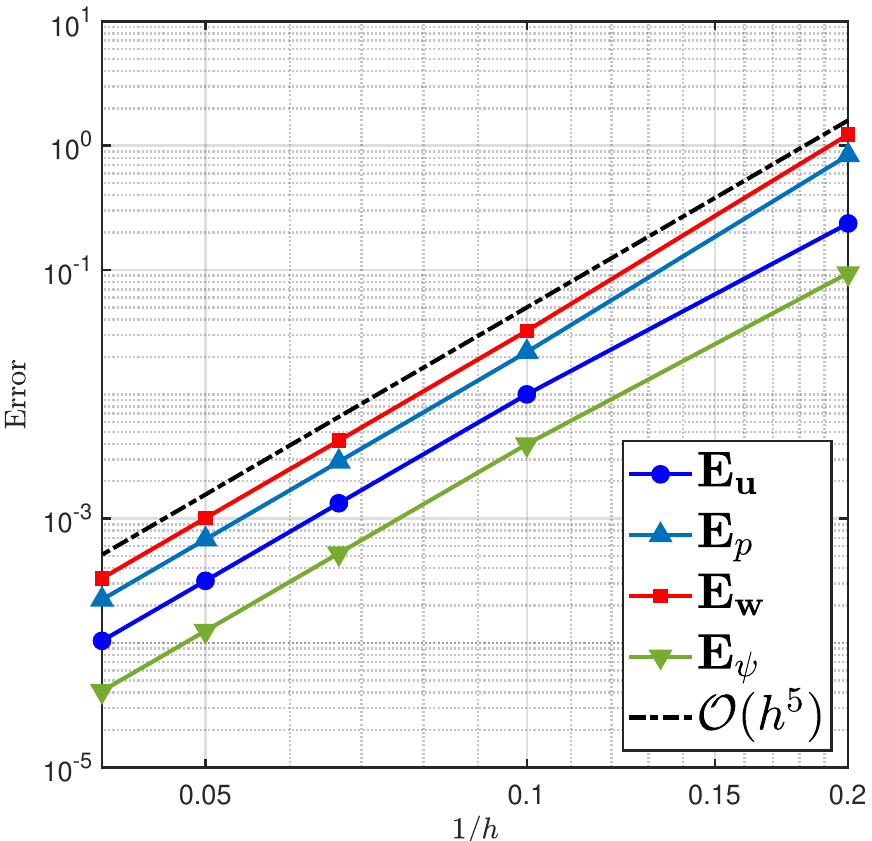}
\end{minipage}
}
\end{tabular}
\caption{Errors for Test 2 on an L-shaped domain with different spline polynomial degree and regularity. a)  $p_{\mathbf v}=2$; b-c) $p_{\mathbf v}=3$; d-e-f) $p_{\mathbf v}=4$;
g-h-i-j) $p_{\mathbf v}=5$}
\label{Lshapefig}
\end{figure}

\begin{table}[tb!]
\centering
\footnotesize
\caption{Convergence errors and orders for Test 3 with different parameter $\lambda, c_0, \kappa$ using $p_{\mathbf v}=3,\,p_p=2$ B-splines on an L-shape domain}
\label{table5}
\begin{tabular}{rccccccccc}
\hline
 $\lambda$ &$1/h$(dof) & $\mathbf E_{\mathbf u}$ & order& $\mathbf E_p$ & order & $\mathbf E_{\mathbf w}$ & order& $\mathbf E_{\psi}$ & order\\[1mm] 
 \multirow{5}*{1e3}&6(1089) &5.42e-02&  &1.35e-01&  &2.73e-01&  &1.34e-01&  \\
 &12(3879)&7.61e-03&2.83&2.18e-02&2.63&2.64e-02&3.37&1.28e-02&3.39\\
 &18(8397) &2.33e-03&2.92&6.62e-03&2.94&7.58e-03&3.07&3.65e-03&3.09\\
 &24(14643)&9.94e-04&2.95&2.81e-03&2.98&3.17e-03&3.02&1.53e-03&3.03\\
 &30(22617)&5.13e-04&2.97&1.44e-03&2.99&1.62e-03&3.01&7.79e-04&3.02\\
\\
 \multirow{5}*{1e8}&6(1089) &5.42e-02& &1.35e-01& &2.73e-01& &1.35e-01& \\
 &12(3879) &7.61e-03&2.83&2.18e-02&2.63&2.64e-02&3.37&1.29e-02&3.39\\
 &18(8397) &2.33e-03&2.92&6.62e-03&2.94&7.58e-03&3.07&3.67e-03&3.09\\
 &24(14643) &9.94e-04&2.95&2.81e-03&2.98&3.17e-03&3.02&1.54e-03&3.03\\
 &30(22617) &5.13e-04&2.97&1.44e-03&2.99&1.62e-03&3.01&7.83e-04&3.02\\
\hline
%
 $c_0$& $1/h$(dof) & $\mathbf E_{\mathbf u}$ & order& $\mathbf E_p$ & order & $\mathbf E_{\mathbf w}$ & order& $\mathbf E_{\psi}$ & order\\[1mm] 
 \multirow{5}*{1e-3}
&6(1089)&2.73e-01&  &9.86e-01& &1.20e+00& &9.38e-02& \\
&12(3879)&4.44e-02&2.62&1.61e-01&2.62&1.28e-01&3.23&1.54e-02&2.61\\
 &18(8397) &1.35e-02&2.93&4.90e-02&2.94&3.71e-02&3.05&4.68e-03&2.93\\
 &24(14643) &5.75e-03&2.98&2.08e-02&2.98&1.56e-02&3.02&1.99e-03&2.98\\
 &30(22617) &2.95e-03&2.99&1.07e-02&2.99&7.97e-03&3.01&1.02e-03&2.99\\
\\
 \multirow{5}*{1e-8}&6(1089)&2.73e-01& &9.87e-01& &1.20e+00& &9.38e-02&\\
 &12(3879) &4.45e-02&2.62&1.61e-01&2.62&1.28e-01&3.23&1.54e-02&2.61\\
 &18(8397)&1.35e-02&2.93&4.90e-02&2.94&3.72e-02&3.05&4.68e-03&2.93\\
 &24(14643)&5.75e-03&2.98&2.08e-02&2.98&1.56e-02&3.02&1.99e-03&2.98\\
 &30(22617) &2.95e-03&2.99&1.07e-02&2.99&7.98e-03&3.01&1.02e-03&2.99\\
\hline
%
 $\kappa$ &$1/h$(dof) & $\mathbf E_{\mathbf u}$ & order& $\mathbf E_p$ & order & $\mathbf E_{\mathbf w}$ & order& $\mathbf E_{\psi}$ & order\\[1mm] 
 \multirow{5}*{1e-3}&6(1089)&4.63e-01& &1.58e+00& &3.40e-03& &2.34e-01&\\
 &12(3879) &5.12e-02&3.18&1.74e-01&3.19&1.84e-04&4.21&2.59e-02&3.18\\
 &18(8397)&1.50e-02&3.03&5.07e-02&3.04&4.88e-05&3.27&7.56e-03&3.03\\
 &24(14643) &6.30e-03&3.01&2.13e-02&3.01&2.00e-05&3.10&3.18e-03&3.01\\
 &30(22617)&3.22e-03&3.00&1.09e-02&3.01&1.01e-05&3.05&1.63e-03&3.01\\
 \\
\multirow{5}*{1e-8}&6(1089)&4.65e-01& &1.59e+00& &3.54e-08& &2.37e-01& \\
 &12(3879) &5.13e-02&3.18&1.74e-01&3.19&1.92e-09&4.21&2.60e-02&3.19\\
 &18(8397) &1.50e-02&3.03&5.08e-02&3.04&5.16e-10&3.23&7.60e-03&3.03\\
 &24(14643) &6.31e-03&3.01&2.14e-02&3.01&2.15e-10&3.05&3.20e-03&3.01\\
 &30(22617) &3.23e-03&3.00&1.09e-02&3.01&1.10e-10&2.99&1.63e-03&3.01\\
\hline
\end{tabular}
\end{table}

$\mathbf{Test\,3.}$ We then test the robustness of our method with respect to the Biot model parameters $\lambda, c_0,$ and $\kappa$, considering again an L-shape domain and exact solution
\begin{align*}
&\mathbf u(x,y,t)=e^t
 \begin{pmatrix}
   \sin(\pi x)\,\cos(\pi y)\\
   -\cos(\pi x)\,\sin(\pi y)
\end{pmatrix}+\frac{e^t}{2\lambda}
 \begin{pmatrix}
  x^2\\
   y^2
\end{pmatrix},
\qquad
\psi(x,y,t)=-\lambda\dvg\,\mathbf{u},
\hspace{3.1cm}\\
&\mathbf{w}(x,y,t)=-\kappa\nabla p,\qquad\qquad\qquad\qquad
p(x,y,t)=\,e^t\,\left(\sin(\pi x)\sin (\pi y)-\frac{4}{3\pi^2}\right).
\end{align*}
We consider isogeometric Taylor-Hood elements with spline degrees $p_{\mathbf v}=3,\,p_p=2$, spline regularity $k_{\mathbf v}=1$.
The results are reported in Table \ref{table5}, where we increase $\lambda$ by 5 orders of magnitude from $10^3$ to $10^8$ (top rows), or decrease by 5 orders of magnitude from $10^{-3}$  to $10^{-8}$  the parameter $c_0$ (middle rows) or $\kappa$ (bottom rows).
In each case, the values of the two other parameters which are not varied are always 1.
For each parameter value, we consider 5 tests with decreasing mesh sizes from $h=1/6$ to $h=1/30$, in order to assess the convergence rate of our method.
The results show that the errors are almost insensitive to the large variations we tested for $\lambda$,  $c_0$, and $\kappa$. In all cases, the method attains the optimal order of convergence $p_{\mathbf v}=3$.

$\mathbf{Test\,4.}$
One of the main motivations for developing the new isogeometric method presented in the previous sections was to overcome spurious oscillations in the pressure variable.
In this last test, we consider the cantilever bracket problem (see \cite{YiSon-Young2013}) and focus on the pressure solution field with spline degrees $p_{\mathbf v} =2,\,p_p = 1$. First, we consider a two-dimensional test on the quarter annulus domain defined in Test 1, with zero displacement boundary condition on the curved edge $\rho=2$, and traction boundary condition  $(\sigma-pI)\mathbf n=(0\,\,-1)^\top$ on the remaining edges, where $\sigma(\mathbf{u})=2\mu\,\varepsilon\,(\mathbf{u})+\,\lambda\,\dvg\,\mathbf{u}\mathbf{I}$. The outer normal component of the flow velocity is assumed to be zero across the entire boundary.
Second, we consider  a three-dimensional test on a horseshoe-like domain as depicted in Figure \ref{OSC3D}. The boundary conditions for this situation are similar  to the previous two-dimensional test. Specifically, we impose the zero displacement boundary condition on the inner U-shaped surface and impose the traction boundary condition $(\sigma-pI)\mathbf n=(0\,\,0\,\,-1)^\top$ on the remaining surface. The outer normal component of the flow velocity remains zero across the entire boundary.

\begin{figure}[htb!]
\centering
\subfigure 
{
\begin{minipage}[t]{4cm}
\centering
\includegraphics[width=4cm]{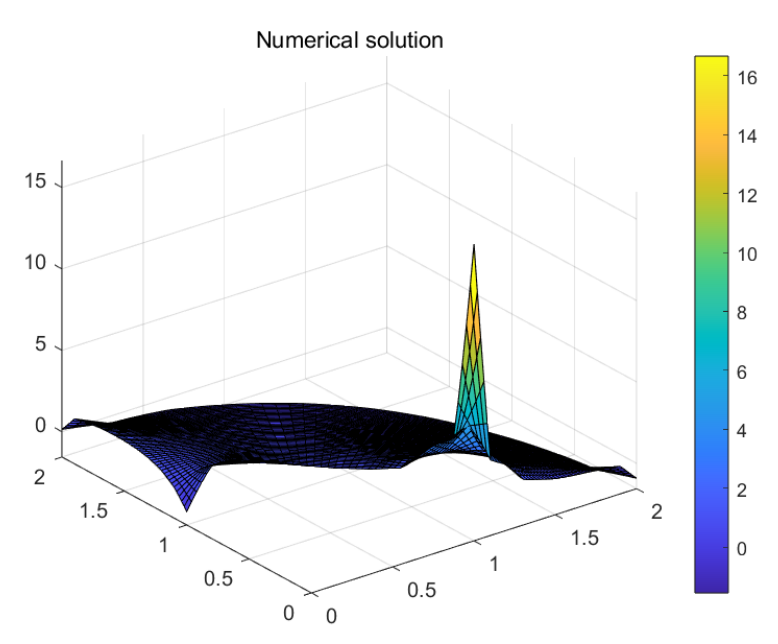}
\end{minipage}
}
\subfigure 
{
\begin{minipage}[t]{4cm}
\centering
\includegraphics[width=4cm]{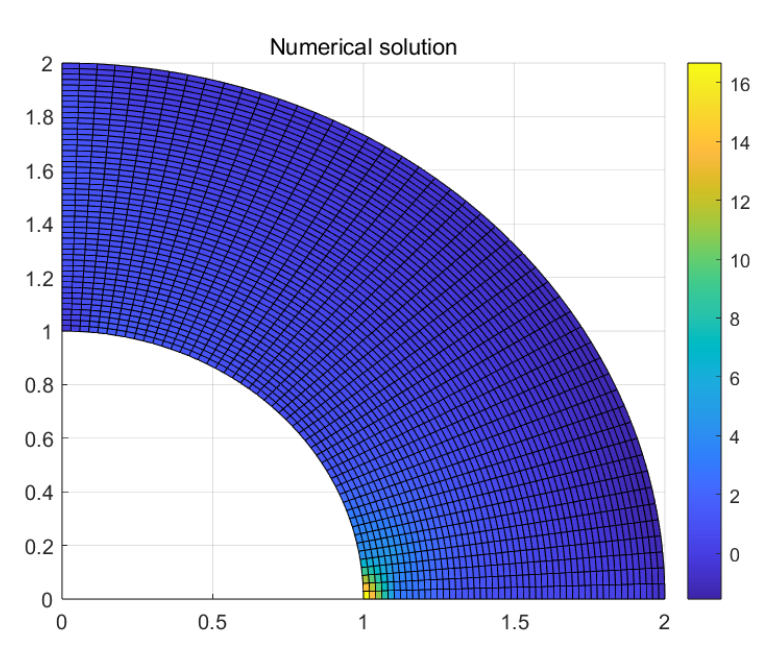}
\end{minipage}
}
 \caption{2D Pressure field for Test 4 at time $T=0.001$ for the cantilever bracket problem; surface plot(left), top view (right)}
 \label{OSC2D}
\end{figure}

\begin{figure}[htb!]
\centering
\subfigure [Solid displacement (first component)]
{
\begin{minipage}[t]{6cm}
\centering
\includegraphics[trim={0 0 0 0},clip,width=5cm]{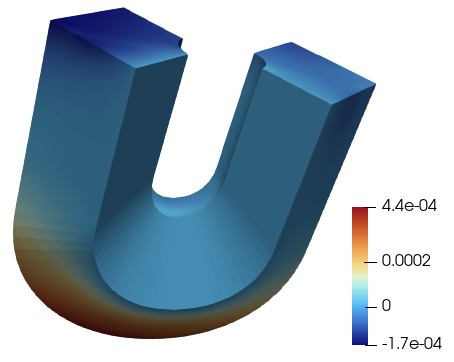}
\end{minipage}
}
\subfigure [Solid pressure]
{
\begin{minipage}[t]{6cm}
\centering
\includegraphics[trim={0 0 0 0},clip,width=5cm]{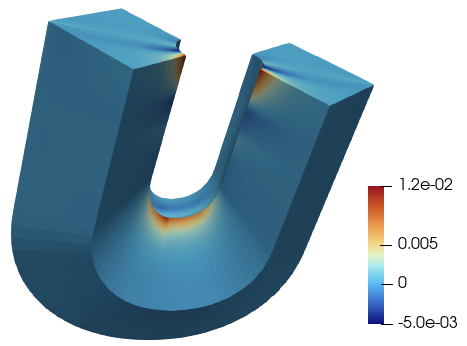}
\end{minipage}
}
\subfigure [Fluid flux (first component)]
{
\begin{minipage}[t]{6cm}
\centering
\includegraphics[trim={0 0 0 0},clip,width=5cm]{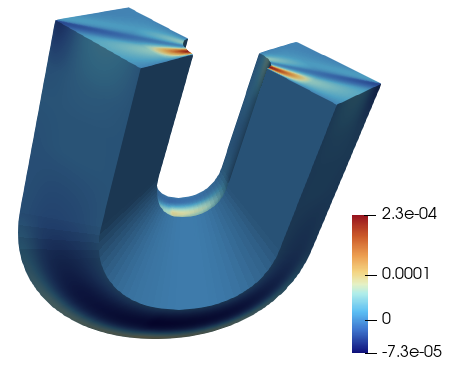}
\end{minipage}
}
\subfigure [Fluid pressure]
{
\begin{minipage}[t]{6cm}
\centering
\includegraphics[trim={0 0 0 0},clip,width=5cm]{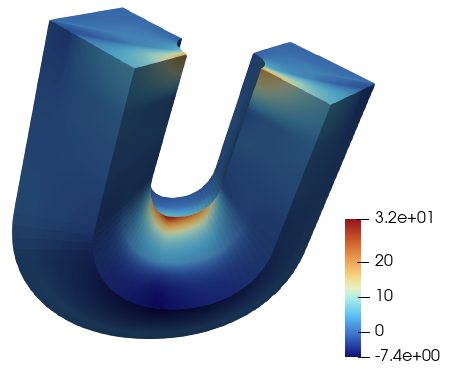}
\end{minipage}
}
\caption{Four solution fields for the cantilever bracket problem of Test 4 in a 3D horseshoe domain at time $T=0.001$.
Solid displacement and pressure (top panels), fluid flux and pressure (bottom panels)}
\label{OSC3D}
\end{figure}

 The material parameters are set to $\alpha=0.93,\,c_0=0,\,\kappa=10^{-7},\,E=10^5,\nu=0.4$.  All other data, including the loads and initial conditions, are assumed to be zero. The pressure solutions for the cantilever Biot test is reported in Figure \ref{OSC2D} for the 2D case and in Figure \ref{OSC3D} for the 3D case. In all cases, no spurious pressure oscillations are present.

$\mathbf{Test\,5.}$  We now test the convergence of the second-order methods in time (Crank - Nicolson and BDF2) described in Section \ref{sec:time}. We consider a Biot model problem  over the square domain $(0,1)^2$, with exact solution
\begin{align*}
&\mathbf u(x,y,t)= (e^t-1)
 \begin{pmatrix}
   \sin(\pi x)\,\sin(\pi y)\\
   \sin(\pi x)\,\sin(\pi y)
\end{pmatrix},
&\psi(x,y,t)=-\lambda\dvg\,\mathbf{u},
\\
&\mathbf{w}(x,y,t)=-\kappa\nabla p,
&p(x,y,t)=(e^{-t}-1)\cos(\pi x),
\end{align*}
with $p_\mathbf v=2$, $k_\mathbf v=0$ and parameters  $c_0=0$, $\kappa=\lambda=\mu=1$.
We focus on the convergence errors in time and, in order to avoid the influence of spatial discretization on the time discretization error, we set the grid size $h = \frac{1}{150}$. Note that, when the mesh size is sufficiently small, Figure \ref{BE} shows the error convergence plot for time discretization using the backward Euler method, demonstrating clear $\mathcal{O}(\Delta t)$-order convergence, while Figure \ref{CN-BDF2}  illustrates the error convergence rates for the Crank-Nicolson (CN) method and the BDF2 method, both achieving $\mathcal{O}({\Delta t}^2)$ convergence. These results show that our discretization schemes maintains optimal convergence rates even for higher-order time discretizations.

\begin{figure}[htp!]
\centering
\includegraphics[width=4.8cm]{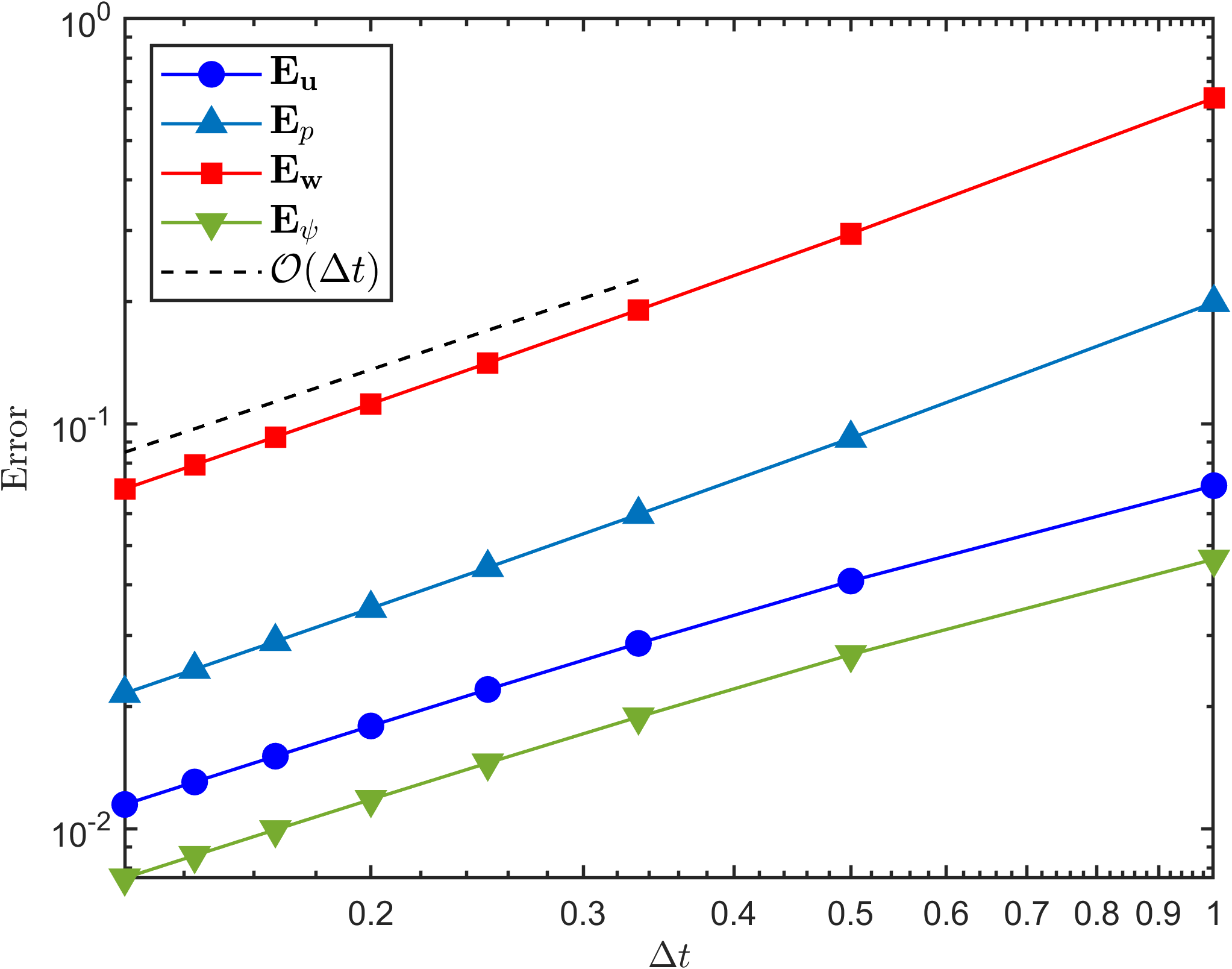}
\caption{Time discretization errors for the backward Euler method.}
\label{BE}
\end{figure}

\begin{figure}[htp!]
\centering
\subfigure [CN]
{\begin{minipage}[t]{5cm}
\centering
\includegraphics[width=4.8cm]{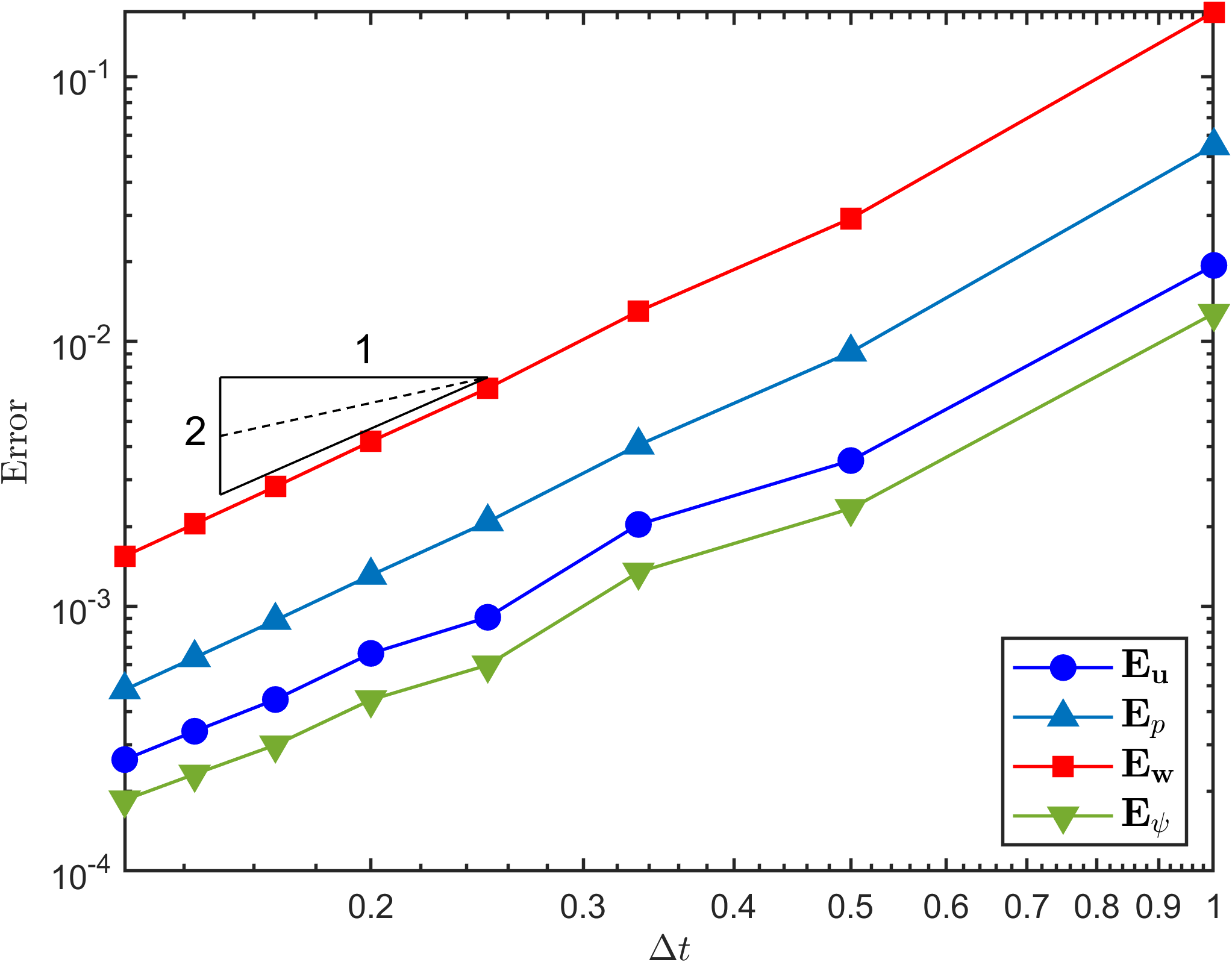}
\end{minipage}
}
\hspace{6mm}\subfigure [BDF2]
{\begin{minipage}[t]{5cm}
\centering
\includegraphics[width=4.8cm]{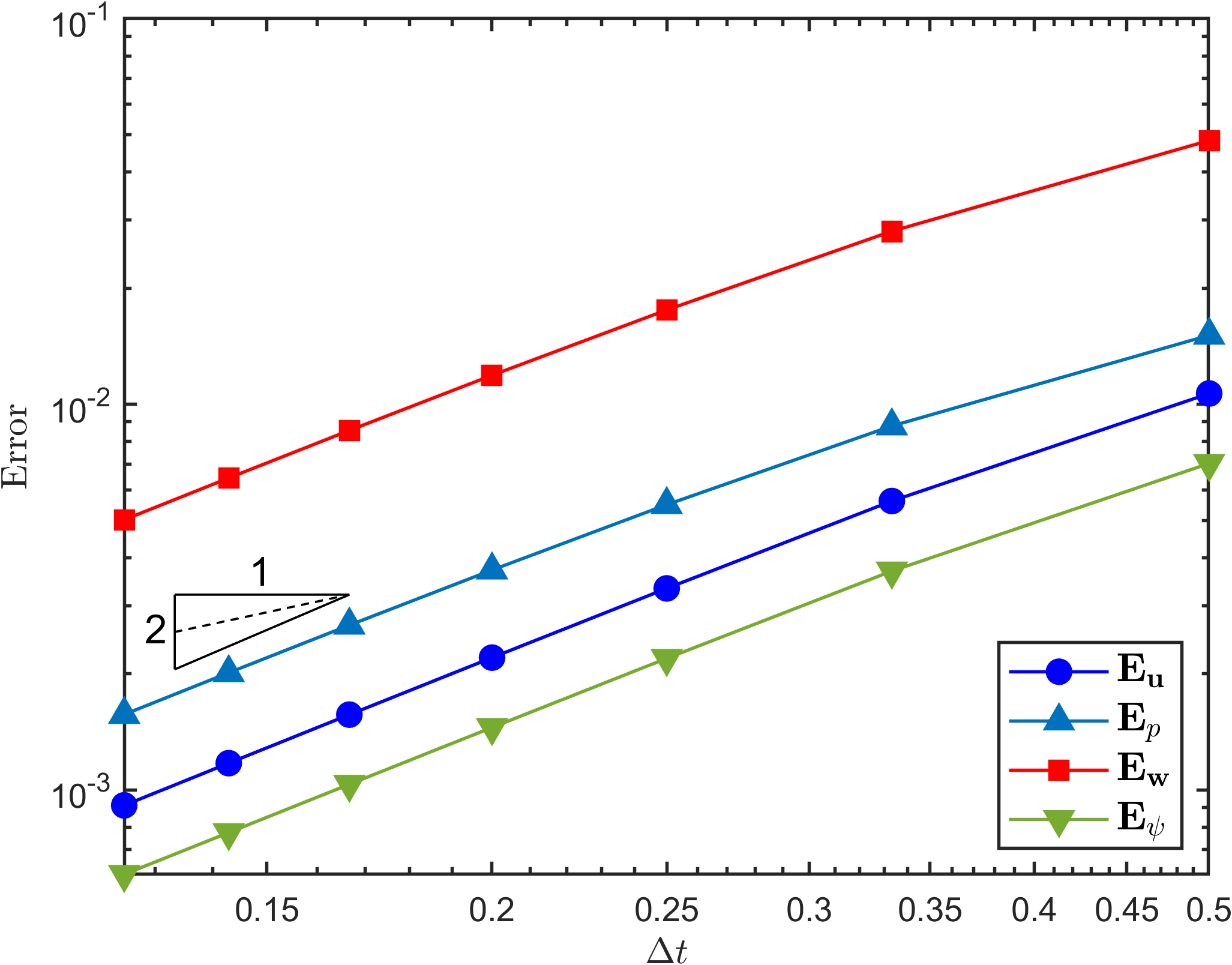}
\end{minipage}
}
\caption{Time discretization errors for the CN and BDF2 methods.}
\label{CN-BDF2}
\end{figure}

$\mathbf{Test\,6.}$ In this test, we perform a numerical comparison between the isogeometric analysis approach proposed in this paper and the finite element method to highlight the superior accuracy of the isogeometric discretization. We solve the four-field formulation of Biot system with the exact solution
\begin{align*}
&\mathbf u(x,y,t)=e^t
 \begin{pmatrix}
   \sin(\pi x)\,\cos(\pi y)\\
   -\cos(\pi x)\,\sin(\pi y)
\end{pmatrix}+\frac{e^t}{2\lambda}
 \begin{pmatrix}
  x^2\\
   y^2
\end{pmatrix},
\qquad
\psi(x,y,t)=-\lambda\dvg\,\mathbf{u},
\hspace{3.1cm}\\
&\mathbf{w}(x,y,t)=-\kappa\nabla p,\qquad\qquad\qquad\qquad
p(x,y,t)=\,e^t\,\left(\sin(\pi x)\sin (\pi y)-\frac{4}{\pi^2}\right),
\end{align*}
and with parameters $\lambda = 10^8$, $c_0 = 0$, $\alpha=\kappa =1$, on the computational domain $(0, 1)^2$. For the finite element method, we employ the Taylor-Hood pair ($P_4 - P_3$) for the variables $(\mathbf{u}, \xi)$ and $(RT_3 - P_3)$ for $(\mathbf{w}, p)$, as described in \cite{boffi2013mixed}. It is well known that, compared to the finite element method, the spline basis functions used in isogeometric analysis exhibit higher smoothness, and that refinement strategies are more flexible. In Fig.\ref{Fig:accuracy}, we present the errors between the exact and numerical solutions for each of the four fields of Biot problem as a function of the degrees of freedom $N$. We employ three different isogeometric refinement strategies and compare them with the standard finite element $h$-refinement (FEM, orange circle). Specifically, we consider
$h$-refinement also for IGA ($h$-IGA, pink triangles), where we reduce the mesh size while maintaining the same polynomial degree $p$ as in the finite element method, but with a higher spline smoothness $k_{\mathbf{v}} = 2$. Then we consider isogeometric
$p$-refinement ($p$-IGA, pink squares), where we fix the spline smoothness to $C^0$ continuity and increase the polynomial degree to expand the discrete space. Finally, we consider isogeometric $k$-refinement ($k$-IGA, blue diamonds), where we increase both the polynomial degree $p$ and the spline smoothness $k$ so as to maintain the maximum regularity  $k = p-1$.
The results of Fig. \ref{Fig:accuracy} show that our isogeometric method exhibits superior convergence accuracy compared to the standard finite element method. In the case of $h$-refinement, FEM and $h$-IGA error plots have the same slope $O(N^{-2})$ (dashed black line), but the $h$-IGA errors are about one order of magnitude smaller. Furthermore, $p$-IGA shows a fast spectral convergence, surpassed only by the even faster $k$-IGA error plots.

\begin{figure}[htb!]
\centering
\subfigure 
{
\begin{minipage}[t]{6cm}
\centering
\includegraphics[trim={0 0 0 0},clip,width=5cm]{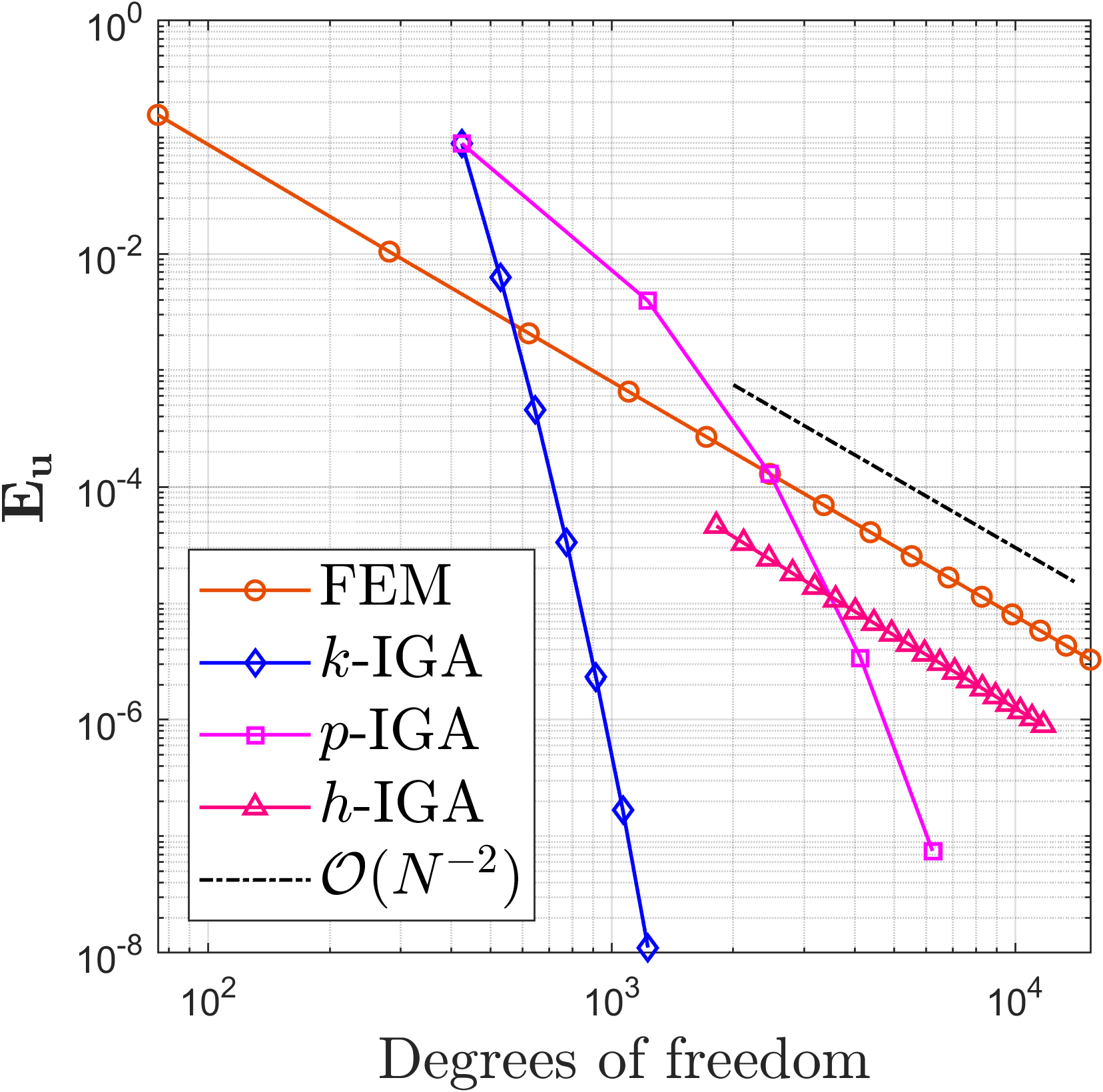}
\end{minipage}
}
\subfigure 
{
\begin{minipage}[t]{6cm}
\centering
\includegraphics[trim={0 0 0 0},clip,width=5cm]{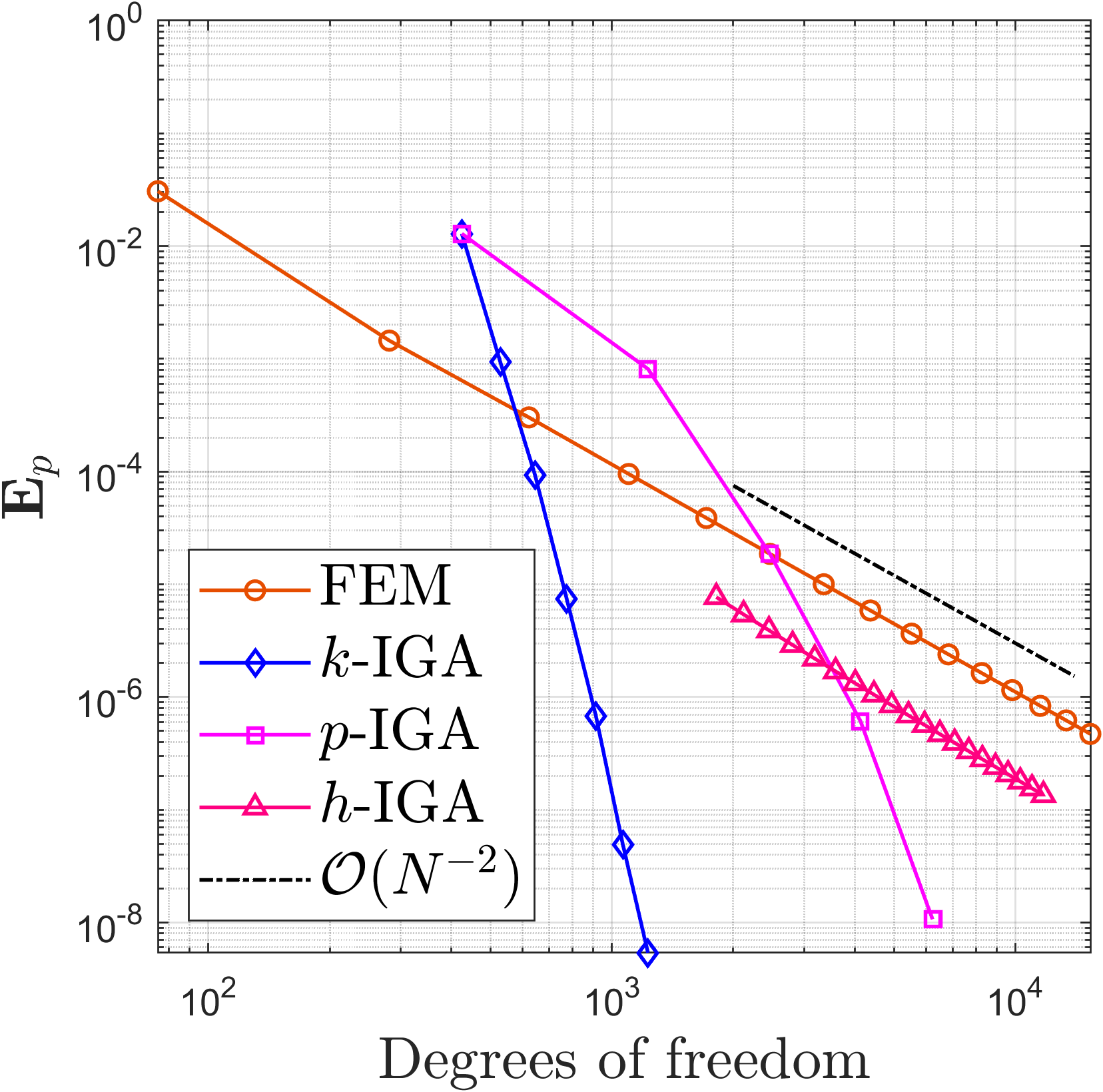}
\end{minipage}
}

\subfigure 
{
\begin{minipage}[t]{6cm}
\centering
\includegraphics[trim={0 0 0 0},clip,width=5cm]{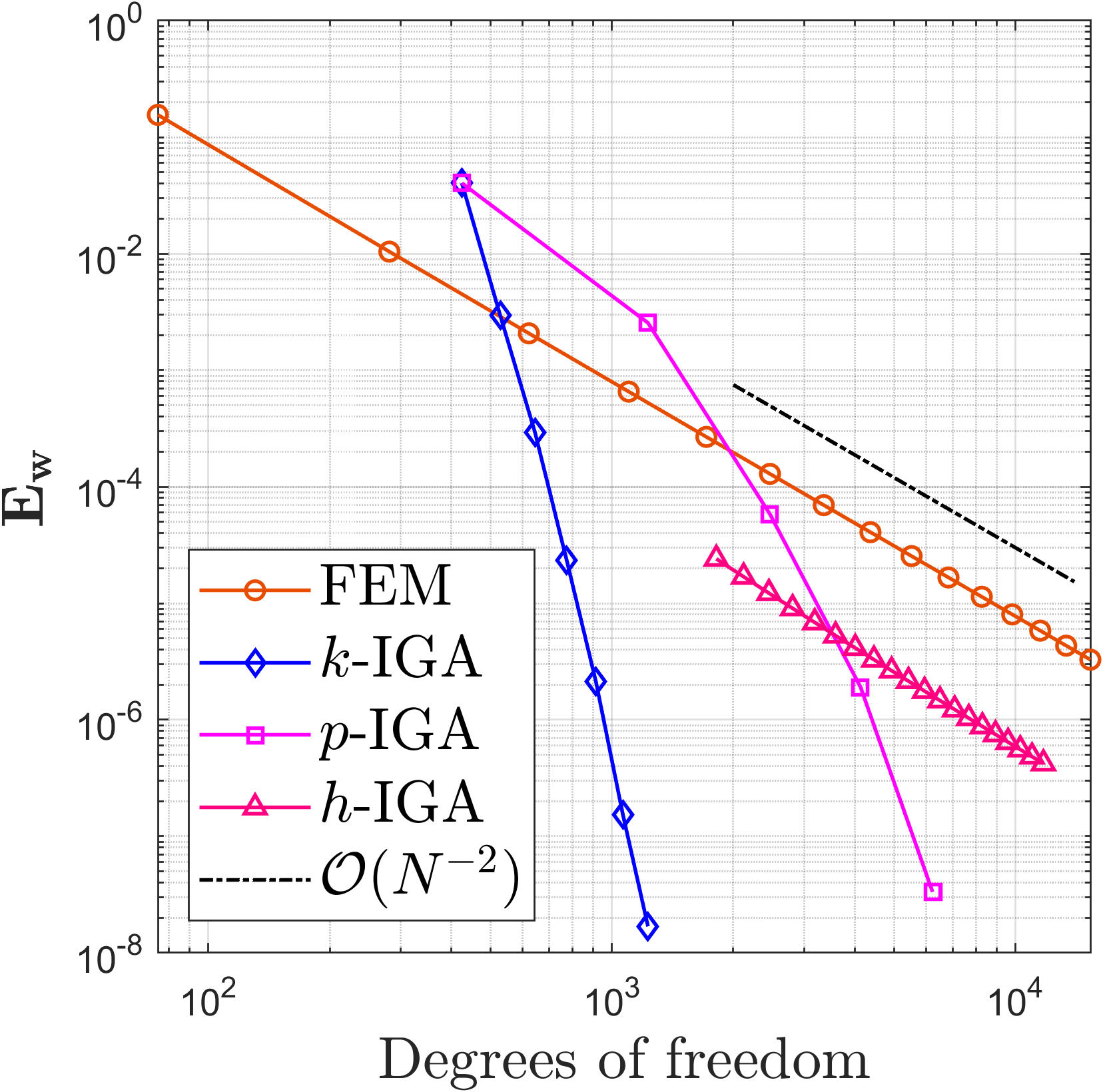}
\end{minipage}
}
\subfigure 
{
\begin{minipage}[t]{6cm}
\centering
\includegraphics[trim={0 0 0 0},clip,width=5cm]{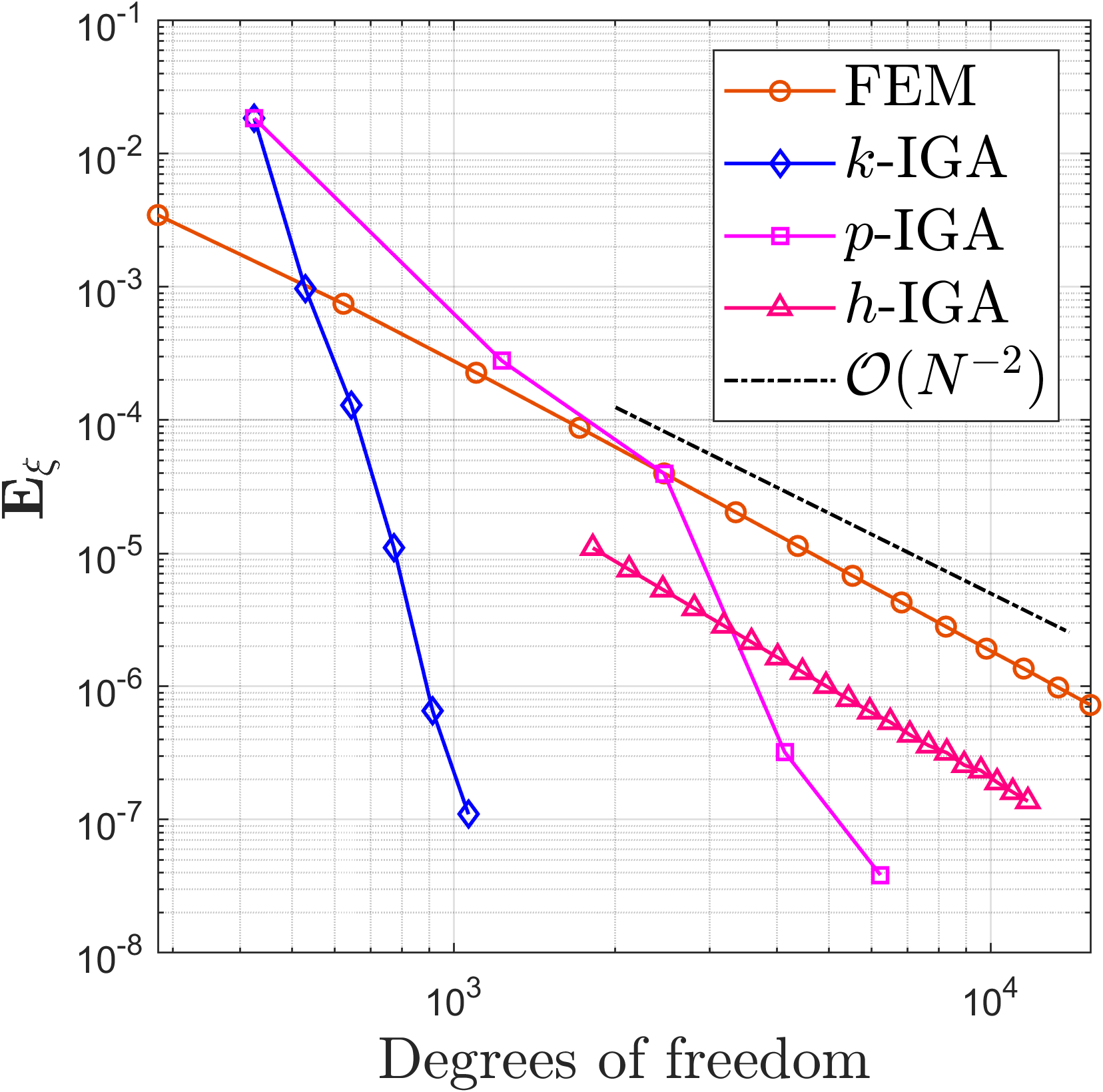}
\end{minipage}
}
\caption{Test 6 convergence errors for the four fields
${\bf E_u}$ (top left), ${\bf E}_p$ (top right),
${\bf E_w}$ (bottom left), ${\bf E}_\xi$ (bottom right), of the Biot's consolidation model.
Isogeometric $h$-IGA, $p$-IGA, $k$-IGA refinements exhibit superior accuracy in each of the four fields compared with standard FEM $h$-refinements.}
\label{Fig:accuracy}
\end{figure}

\section{Conclusions}\label{conclusions}
In this paper, we have introduced a novel isogeometric method for solving Biot's consolidation model, utilizing a four-field formulation. Mixed isogeometric spaces based with B-splines basis functions are employed in the space discretization of the four fields (solid displacement, solid pressure, fluid flux, fluid pressure). We proved optimal error estimates that are robust with respect to the material parameters, in particular the Lam\'e constant $\lambda$ and the storage coefficient $c_0$, effectively preventing both elasticity locking for nearly incompressible materials and pressure oscillations. We have carried out
numerical experiments in two and three dimensions that confirm our theoretical error estimates. In particular, we have shown that the proposed isogeometric Biot discretization attains optimal high-order convergence rates in both the mesh size $h$ and the spline polynomial degree $p$, in each of the four fields,  assessing also the numerical performance of our isogeometric method with respect to the spline regularity and material parameters.

\section{Acknowledgements}
Hanyu Chu has been supported by National Natural Science Foundation of China (Grants. 12071227, 12201547) and China Scholarship Council 202306860055.
Both authors have been supported by grants of MUR (PRIN P2022B38NR and PRIN 202232A8AN\_02) and INdAM--GNCS.

\hfill
\bibliography{Reference}
\bibliographystyle{plain}

\end{document}